\documentclass[a4paper,12pt]{article}
\usepackage{amssymb}
\usepackage{amsthm}
\usepackage{amsxtra}
\usepackage{amsfonts}
\usepackage{xr}
\usepackage{color}

\definecolor{orange}{rgb}{1,0.5,0}

\definecolor{secret}{rgb}{0,0.5,1}

\newenvironment{Zo}{\begingroup\color{red}}{\endgroup}

\newtheorem{thm}{Theorem}[section]
\newtheorem{cor}[thm]{Corollary}
\newtheorem{lem}[thm]{Lemma}
\newtheorem{prop}[thm]{Proposition}

\theoremstyle{definition}
\newtheorem{para}[thm]{}

\newtheorem{defn}[thm]{Definition}
\newtheorem{notation}[thm]{Notation}
\newtheorem{rem}[thm]{Remarks}
\newtheorem{remark}[thm]{Remark}
\newtheorem{fact}[thm]{Fact}

\newcommand{\si}{\sigma}
\newcommand{\prf}{\smallskip\noindent{\it        Proof}. }
\newcommand{\call}{{\mathcal L}}
\newcommand{\nat}{{\mathbb  N}}

\newcommand{\inv}{^{-1}}

\newcommand{\rest}{{\lower       .25     em      \hbox{$\vert$}}}

\newcommand{\zee}{{\mathbb  Z}}

\newcommand{\acl}{{\rm acl}}

\newcommand{\aff}{{\mathbb A}}

\newcommand{\calu}{{\mathcal U}}

\newcommand{\rat}{{\mathbb Q}}
\newcommand{\ga}{{\mathbb G}_a}
\newcommand{\gm}{{\mathbb G}_m}

\newcommand{\fix}{{\rm Fix}}

\newcommand{\calb}{{\mathcal B}}

\newcommand{\dcl}{{\rm dcl}}

\newcommand{\calc}{{\mathcal C}}
\newcommand{\Ker}{{\rm Ker}\,}

\font\helpp=cmsy5
\newcommand{\semdp}
{\hbox{$\times\kern-.23em\lower-.1em\hbox{\helpp\char'152}$}\,}

\newcommand{\dnfo}{\,\raise.2em\hbox{$\,\mathrel|\kern-.9em\lower.35em\hbox{$\smile$}
$}}
\newcommand{\dfo}{\;\raise.2em\hbox{$\mathrel|\kern-.9em\lower.35em\hbox{$\smile$}
\kern-.7em\hbox{\char'57}$}\;}

\newcommand{\qfcb}{\hbox{qf-Cb}}

\newcommand{\cald}{{\mathcal D}}
\newcommand{\Aut}{{\rm Aut}}

\newcommand{\dcf}{{\rm DCF}_m}

\newcommand{\acfa}{{\rm ACFA}}
\newcommand{\dcfa}{{\rm DCF}_m{\rm A}}

\newcommand{\SU}{{\rm SU}}
\newcommand{\scl}{{\rm scl}}

\newcommand{\vlabel}{\label}
\begin{document}
\title{Groups definable in partial differential fields with an automorphism}

\author{Ronald F. Bustamante Medina \footnote{Escuela de
    Matem\'atica-CIMPA, Universidad de Costa Rica.} \footnotemark[4],
  \\Zo\'e Chatzidakis \footnote{IMJ-PRG (UMR 7586), CNRS, Universit\'e Paris-Cit\'e.} \footnote{Partially
    supported by ANR-17-CE40-0026 (AGRUME) and ANR-DFG AAPG2019 (GeoMod)} \footnotemark[4]
 \; and \\Samaria Montenegro  \footnotemark[1] \footnote{Partially supported by Proyecto Semilla 821-C0-464, Vicerrector\'ia in Investigaci\'on, UCR.}}




\maketitle

\begin{abstract}  
In this paper we study groups definable in existentially closed partial
differential fields of characteristic $0$ with  an automorphism which
commutes with the derivations. In particular, we study Zariski dense
definable subgroups of simple algebraic groups, and show an analogue of P.J.~Cassidy's result for partial differential fields. We also show that
these groups have a smallest definable subgroup of finite index. 

  
\end{abstract}

\section{Introduction}

Fields with operators appear everywhere in mathematics, and are
particularly present in areas close to algebra. The development of
differential and difference algebra dates back to J. Ritt (\cite{Ritt})
in the 1950's, and was then further expanded by E.~R.~Kolchin (\cite{Ko}, \cite{Ko2}) and R. Cohn (\cite{Co}) in the 1960's. 
The study of differential and difference fields has been important in mathematics since the 1940’s and has  applications in many areas of mathematics. 

One can also mix the operators, this gives the notion of differential-difference fields, i.e., a field equipped with commuting derivations and automorphisms. These fields were first studied from the point of view of algebra by Cohn in \cite{Co70}.

Model theorists have long been interested in fields with operators,
until recently mainly on  fields of characteristic $0$ with one or
several commuting derivations (ordinary or partial differential fields), and
on fields with one automorphisms (difference fields). The first author also
started in \cite{Bu1} the model-theoretic study of the existentially closed
difference-differential fields of characteristic $0$, around 2005 (one derivation, one automorphism). Then work of D. Pierce  (\cite{Pier}) and of O. Le\'on-S\'anchez (\cite{LS}) brought back the model-theory of differential fields with several commuting derivations in the forefront of research in the area, as well as when a generic automorphism is added to these
fields. However, unlike in the pure ordinary differential case and in the pure
difference case, little is known on the possible interactions between definable subsets of existentially closed differential fields with several derivations, nor with an added automorphism.

In this paper we study groups definable in  existentially closed
differential-difference fields.  We were motivated by the following
result of P.~J.~Cassidy (Theorem~19 in \cite{Ca2}; we phrase it differently):

\medskip\noindent
{\bf Theorem}.  {\em Let $\calu$ be a differentially closed field of
characteristic $0$ (with $m$ commuting derivations), let $H$
  be a simple algebraic group defined and split over $\rat$, and $G\leq H(\calu)$  a
  $\Delta$-algebraic subgroup of $H(\calu)$ which is Zariski dense in $H$. Then $G$ is definably isomorphic to $H(L)$,
  where $L$ is the constant field of a set $\Delta'$ of commuting derivations. Furthermore, the isomorphism is
  given by conjugation by an element of $H(\calu)$. } 

  \medskip\noindent
She has similar results for Zariski dense $\Delta$-closed subgroups of
semi-simple algebraic groups. 
A version of her result for (existentially closed) difference
fields was also proved by Z. Chatzidakis, E. Hrushovski and Y. Peterzil (Proposition~7.10 of \cite{CHP}): 

\medskip\noindent
{\bf Theorem}. {\em 
Let $(\calu, \sigma)$ be a model of $\acfa$. Let  $H$ be a simple algebraic group defined over $\calu$, and let $G$ be a Zariski dense definable subgroup
of $H(\calu)$. If $SU(G)$ is infinite then $G = H(\calu)$. 
If $SU(G)$ is finite, there are an isomorphism $f: H \rightarrow H' $ of algebraic groups, and integers $m>0$ and $n$ such that some subgroup of $f(G)$
of finite index is conjugate to a subgroup of $H'(\fix (\sigma^m Frob^n))$. In particular, the generic types of $G$ are non-orthogonal to the formula
$\sigma^m(x)= x^{p^{-n}}$. If $H$ is defined over $\fix(\sigma)^{alg}$,  then we may take $H = H'$ and $f$ to be conjugation by an element of $H(\calu)$.}

\medskip\noindent
In this paper, we 
 generalise
 Cassidy's results to the theory $\dcfa$, the model companion of
   the theory of fields of characteristic $0$ with $m$ 
   derivations and an automorphism which commute, and one of our main
   results is: \\[0.1in]
{\bf Theorem \ref{prop1}}. {\em Let $\calu$ be a model of $\dcfa$, let
  $H$ be a simple algebraic group defined and split 
  over $\rat$, and $G$ a definable subgroup
  of $H(\calu)$ which is Zariski dense in $H$.\\
  Then $G$ has a definable subgroup $G_0$ of finite index, the
        Kolchin closure of which   is conjugate (say by an element $g\in
        H(\calu)$) to $H(L)$, where $L$ 
         is  the field of
         constants of a set of definable commuting derivations. Furthermore,
         either $G_0^g=H(L)$, or $G_0^g\subseteq H(\fix(\si^\ell)(L))$ for some
         integer $\ell\geq 1$. 
In the latter case, if $H$ is centerless, we are able to describe
precisely the subgroup $G_0^g$ as $\{g\in H(L)\mid \si^r(g)=\varphi(g)\}$
for some integer $r$ and algebraic automorphism
$\varphi$ of $H(L)$. }\\[0.05in] 
We have  analogous results for Zariski dense definable subgroups of semi-simple
centerless algebraic groups (Theorem \ref{prop2semi}). Using an
isogeny result (Proposition \ref{thm1}), and introducing the appropriate
notion of {\em definably quasi-(semi-)simple} definable group (see
Definition \ref{definably-quasi-simple}), gives then
slightly more general results in the definably quasi-simple case, see
Corollary~\ref{MainTheorem}. The results in the (definably-quasi-)
semi-simple case require a little more work and are less easy to state, see Theorem \ref{prop2semi}
for a precise statement. \\[0.1in]
Inspired by results of Hrushovski and A. Pillay on groups definable in
pseudo-finite fields, we then endeavour to show that definable groups
which are definably
quasi-semi-simple have a smallest definable subgroup of finite
index (this smallest definable subgroup is called the {\em connected
  component}). This is done in Corollary \ref{connected2}, and follows from
several intermediate results. We first show the result for Zariski dense definable
subgroups of a simply connected algebraic simple group $H$,  give a
precise description of the connected component (Theorem 
\ref{connected1}), and show that every definable Zariski dense subgroup
of $H(\calu)$ is quantifier-free definable.  We then show the existence
of a smallest definable subgroup of finite index for an arbitrary simple
algebraic group $H$  (Theorem \ref{prop2}), to finally reach the conclusion.  
Part of the study involves giving a description of definable subgroups
of (arbitrary) algebraic groups and we obtain the following result, of independent
interest: \\[0.05in]
{\bf Theorem \ref{sbgp}}. {\em Let $H$ be an algebraic group, $G\leq
  H(\calu)$  a Zariski dense definable subgroup. Then there are an
  algebraic group $H'$, 
    a quantifier-free definable subgroup $R$ of $H'(\calu)$, together with a
  quantifier-free definable  $f:R\to G$, with $f(R)$ contained and of finite index in $G$, and $\Ker(f)$
  finite central in $R$.}\\[0.05in]
We also show how this result  extends to existentially closed difference
field of any characteristic, and state the appropriate generalization to
$\fix(\si^\ell)$ (Theorems \ref{ACFAp} and \ref{F-ell-2}).

\bigskip\noindent
The paper is organised as follows. Section 2 contains the algebraic and
model-theoretic preliminaries. Section 3 introduces the notions of
definably quasi-(semi-)simple groups and shows the isogeny result
(\ref{thm1}). Section 4 contains the main results of the paper:
description of Zariski dense definable subgroups of simple and
semi-simple algebraic groups (\ref{prop1}, \ref{MainTheorem} and
\ref{prop2semi}). Section 5 proves the above mentioned 
Theorem~\ref{sbgp} 
and shows that definably quasi-semi-simple definable groups have a
definable connected component. 

\section{Preliminaries}

This section is divided in five subsections: 2.1 - Differential and
difference algebra; 2.2 - Model theory of differential and difference
fields; 2.3 - The results of Cassidy; 2.4 - Quantifier-free canonical
bases; 2.5 - Some well-known group-theoretic results. 

\noindent{\bf Notation and conventions:} 
{\bf All  rings are commutative, all
    fields  are commutative of characteristic 0.}

\noindent  
If $K$ is a field, then $K^{alg}$ denotes an algebraic closure of
  $K$ (in the sense of the theory of fields). 

\subsection{Differential and difference algebra}
\begin{defn} 
For more details, please see \cite{Ko2},
\cite{Co} and \cite{Co70}.
\begin{enumerate}

\item Recall that a \emph{derivation} on  a ring $R$ is a map $\delta: R
  \to R$ which satisfies  $\delta(a +b) = \delta(a) + \delta (b)$ and 
$\delta(ab) = a \delta (b) + \delta(a) b$ for all $a, b \in R$.  

\item A \emph{differential ring}, or \emph{$\Delta$-ring}, is a ring
  equipped with a set $\Delta=\{\delta_1,\ldots,\delta_m\}$  of
  commuting derivations. A {\em differential field} is a differential
  ring which is a field. 

\item A {\em difference ring} is a ring equipped with a distinguished
  automorphism, which we denote by $\si$. (This differs from the usual
  definition which only requires $\si$ to be an endomorphism.) A {\em difference field} is a
  difference ring which is a field.

  \item A {\em difference-differential ring} is a differential ring
    equipped with an automorphism $\si$ (which commutes with the
    derivations). A {\em difference-differential field} is a
    difference-differential ring which is a field.

\end{enumerate}
\end{defn}
 
\begin{notation} 
  \begin{enumerate}

  \item
    If $K$ is a difference field, then $\fix(\si)(K)$, or 
    $\fix(\si)$ if there is no ambiguity,  denotes
    the {\em fixed field} of $K$, $\{a\in K\mid \si(a)=a\}$. 

\item Let $K\subset\calu$ be difference-differential fields, and
  $A\subset \calu$. Then $K(A)_\Delta$ denotes the differential field
  generated by $A$ over $K$, $K(A)_\si$ the difference field generated
  by $A$ over $K$, and $K(A)_{\si,\Delta}$ the difference-differential
  field generated by $A$ over $K$. (Note that we require $K(A)_\si$ and
  $K(A)_{\si,\Delta}$ to be closed under $\si\inv$.)

\end{enumerate} 
\end{notation}

\medskip\noindent
{\bf \large Polynomial rings and the corresponding ideals and topologies}

\begin{defn} Let $K$ be a difference-differential ring, $y=(y_1,\ldots,y_n)$ a tuple of
indeterminates. 
\begin{itemize}
	\item 
Then $K\{y\}$ (or $K\{y\}_\Delta$)  denotes the ring of
polynomials in the variables $\delta_1^{i_1}\cdots\delta_m^{i_m}y_j$,
where $1\leq j\leq n$, and the superscripts $i_k$ are non-negative
integers. It becomes naturally a  differential ring, by setting
$\delta_k(\delta_1^{i_1}\cdots\delta_m^{i_m}y_j)=\delta_1^{j_1}\cdots\delta_m^{j_m}y_j$,
where $i_\ell=j_\ell$ if $\ell\neq k$, and $j_k=i_k+1$. The elements of
$K\{y\}$ are called {\em differential polynomials}, or {\em
  $\Delta$-polynomials}.

\item  $K[y]_\si$ denotes the ring of polynomials in the variables
  $\si^i(y_j)$, $1\leq j\leq n$, $i\in\zee$, with the obvious action of
  $\si$; thus it is also a difference ring. They are called {\em
    difference polynomials}, or {\em $\si$-polynomials}. 
\item 
$K\{y\}_\si$ denotes the ring of
polynomials in the variables $\si^i \delta_1^{i_1}\cdots\delta_m^{i_m}y_j$, with the obvious action of $\si$
and derivations. They are called {\em
    difference-differential polynomials}, or {\em $\si$-$\Delta$-polynomials}. 

\item A {\em $\Delta$-ideal} of a differential ring $R$ is an ideal
  which is closed under the derivations in $\Delta$ and it is
called {\em linear} if it is generated by homogeneous linear
$\Delta$-polynomials. 
	
  Similarly, a {\em $\si$-ideal} $I$ of a
    difference ring $R$ is an ideal closed under $\si$; if it is also
    closed under $\si\inv$, we will call it {\em reflexive}; if whenever
    $a\si^n(a)\in I$, then $a\in I$, it is {\em perfect}.
    Finally, a {\em $\si$-$\Delta$-ideal} is an ideal which is closed under
      $\si$ and $\Delta$. 
\end{itemize}
\end{defn} 

\begin{rem}
As with the Zariski topology, if $K$ is a difference-differential field,
the set of zeroes of differential
polynomials, $\si$-polynomials and $\si$-$\Delta$-polynomials in some
$K^n$ are the basic closed sets of a Noetherian topology on $K^n$, see
Corollary 1 of Theorem III in \cite{Co70}.  We will call these sets {\em $\Delta$-closed} (or {\em Kolchin
closed}, or {\em $\Delta$-algebraic}), {\em $\si$-closed/algebraic} and
{\em $\si$-$\Delta$-closed/algebraic} respectively. These
topologies are called the {\em Kolchin topology} (or {\em
  $\Delta$-topology}),  {\em $\si$-topology} and {\em $\si$-$\Delta$-topology}
  respectively. There are natural notions of closures and of irreducible
  components. 
\end{rem}

\begin{fact} \vlabel{sit}
Recall that if $K$ is a differential subfield of $\calu$, then  
  $K\Delta$ denote the $K$-vector space generated by the
 elements of $\Delta$.

      \begin{enumerate}
      
\item
 (\cite{Ko2}, Proposition  0.5.7)
Every commuting linearly independent subset of
$\calu \Delta$ is a subset of a commuting basis of 
$\calu\Delta$.
\item (\cite{Ko2}, Corollary to Proposition 0.5.8) Let
  $\Delta'\subset\calu \Delta$ be a set of commuting  derivations. Then
  the $\Delta'$-differential field $\calu$ is differentially closed. 
\item (\cite{Ko2}, 0.8.13) (Sit) Let A be a perfect $\Delta$-ideal of the $\Delta$-algebra $\calu\{y\}$.
A necessary and sufficient condition that the set of zeroes $Z(A)$ of
$A$ be a subring of $\calu$, is that there exist a vector subspace and Lie
subring $\cald $ of $\calu \Delta$ such that $A = [\cald y]$ (the $\Delta$-ideal generated by
all $Dy$, $D\in \cald$). When this is the case, there exists a commuting
linearly independent subset $\Delta'$ of $\calu \Delta$ such that $Z(A)$ is the field of
absolute constants of the $\Delta'$-field $\calu$.

      \end{enumerate}
      \end{fact}

\noindent
Item (1) is valid under weaker hypotheses

\begin{defn} For short, we will call {\em field of constants in $\calu$}
  a subfield of $\calu$, which is the field of constants of some subset
  $\Delta'\subset\calu\Delta$ of commuting derivations. 
  \end{defn}

\begin{rem} \vlabel{defga} Recall that we are in
characteristic $0$, this result is false in positive characteristic. We
let $K$ be a differential subfield of the differentially closed field $\calu$. Consider the commutative monoid $\Theta$ (with $1$)
generated by $\delta_1,\ldots,\delta_m$, and let $K\Theta$ be the
$K$-vector space with basis $\Theta$. It can be made into a ring,
using the commutation rule $\delta_i \cdot a=a\delta_i
+\delta_i(a)$, $i=1,\ldots,m$. Each element $f$ of $\calu\Theta$ defines a {\em linear
  differential operator} $L_f:\ga\to\ga$, defined by $a\mapsto f(a)$. One
  has $L_{f\cdot g}=L_f\circ L_g$. Every $\Delta$-closed subgroup of
  $\ga(\calu)$ is then 
  defined as the set of zeroes of a finite set of linear differential operators in $\calu\Theta$, and for
  $n\geq 1$, every $\Delta$-closed subgroup of $\ga^n(\calu)$ is defined by
  conjunctions of equations of the form $L_1(x_1)+\cdots +L_n(x_n)=0$,
  with the $L_i$ in $\calu\Theta$, and with the $L_i$ in $K\Theta$ if
  the subgroup is defined over $K$,  see e.g. Proposition~11 in
  \cite{Ca1}, or Proposition 0.8.12 in \cite{Ko2}. Sets of elements of
  $K\Theta$ generate what is called a {\em linear $\Delta$-ideal}. The
  following result is  certainly well-known, but we did not know of a reference. 
  \begin{lem} \vlabel{lemTheta}
Let $S$ be a $K$-subspace of $K\Theta$, and assume that it is closed
under left multiplication by the $\delta_i$, $i=1,\ldots,m$, and that it does not contain $1$. Then
the differential ideal $I$ generated by the set $$S(x):=\{f(x)\mid f\in S\}\subset
K\{x\}_\Delta$$ does not contain $x$ and is prime. \end{lem} 

\prf  Note that $I$ is simply the $K\{x\}_\Delta$-module generated by $S(x)$,
i.e., an element of $I$ is a finite $K\{x\}_\Delta$-linear combination
of elements of $S(x)$.  Every element $f$ in $K\{x\}_\Delta$ can be written uniquely as
$f_0+f_1+f_{>1}$, with $f_0$ the constant term, $f_1$ the sum of the
linear terms, and $f_{>1}$ the sum of terms of $f$ of total degree $\geq
2$. Note that $(f+g)_i=f_i+g_i$ for $i\in \{0,1,>1\}$. Moreover
$$(fg)_0=f_0g_0,\ \ (fg)_1 = f_0g_1+f_1g_0, \ \
(fg)_{>1}=f_{>1}g+fg_{>1}+f_1g_1.$$
Since all elements of $S(x)$ have  degree
$\geq 1$, it follows that all elements of $I$ have constant term
$0$. This easily implies that if $f\in I$, then $f_1\in S(x)$: as $f$ is a $K\{x\}_\Delta$-linear combination of elements of
$S(x)$, and as the elements of $S(x)$ have degree $1$, it follows that $f_1$ is a
$K$-linear combination of elements of $S(x)$ of degree $1$, i.e.,
belongs to $S(x)$. But $1\notin S$ implies $x\notin S(x)$, and therefore
$x\notin I$. \\
The primeness of $I$
follows from the fact that $I$ is generated by linear differential 
polynomials, so that, as a ring, $K\{x\}_\Delta/I$ is isomorphic to a
polynomial ring (in maybe infinitely many indeterminates) over
$K$. \qed

\end{rem}

\subsection{Model theory of differential and difference fields} \label{prel}

\begin{notation}
We consider the language $\call$ of rings, let
$\Delta=\{\delta_1,\ldots,\delta_m\}$. We define $\call_\Delta=\call\cup
\Delta$, $\call_\si=\call\cup \{\si\}$ and $\call_{\si,
  \Delta}=\call_\Delta\cup \{\si\} $  where the $\delta_i$ and $\si$ are
unary function symbols. 
\end{notation}

\para {\bf The theory $\dcf$}\\
The model theoretic study of differential fields (with one derivation, in
characteristic $0$) started with the work of A. Robinson (\cite{Ro})
 and
of L.~C.~Blum (\cite{Bl}). 
For several commuting derivations, 
T.~McGrail showed in \cite{MG}  that the $\call_\Delta$-theory of differential fields of characteristic zero
with $m$ commuting derivations has a model companion, which we denote by
$\dcf$. The $\call_\Delta$-theory $\dcf$ is complete, $\omega$-stable and eliminates
quantifiers and imaginaries. Its models are called {\em differentially
  closed}. Differentially closed fields had appeared earlier in the work
of Kolchin (\cite{Ko}).  From now on till \ref{P-fund}, definable will
mean $\call_\Delta$-definable in $\calu$, maybe with parameters. \\
An interesting corollary of Lemma \ref{lemTheta} is the following:
\begin{thm}\vlabel{autL}
Let $\calu\models \dcf$, $L\leq \calu$ be a {\em subfield of constants}. Then there is no  definable non-trivial automorphism of $L$. 
\end{thm}

\prf Suppose that $\psi:L\to L$ is a definable isomorphism. The graph of $\psi$ defines an additive subgroup $S$ of
$L\times L\leq \calu\times \calu$. \\
By Remark \ref{defga} there are linear differential polynomials $F_i(x)$
and $G_i(y)$, 
$i=1,\ldots, s$, such that 
$$S=\{(x,y)\in L\times L\mid F_1(x)=G_1(y),\ldots,F_s(x)=G_s(y)\}.$$
Let $H_1,\ldots,H_r$ be the linear combinations of the
  $\delta_i$ such that $H_1(x)=H_2(x)=\cdots=0$ defines the field
  $L$. Then, because $S$ is the graph of an
automorphism of $L$, we have
$$\bigcap_{i=1}^s\Ker(F_i)\cap
  L=\{0\}=\bigcap_{i=1}^s\Ker(G_i)\cap L.$$ Hence, $x$
belongs to the differential ideal generated by the $F_i(x)$ and the $H_j(x)$, and this
implies (see Lemma \ref{lemTheta}) that there are linear differential polynomials $L_1,\ldots,L_{s+r}$
such that $\sum_{i=1}^s L_i(F_i(x))+\sum_{j=1}^rL_{s+j}(H_j(x))=x$; letting $G(y)=\sum_{i=1}^s L_i(G_i(y))+\sum_{j=1}^rL_{s+j}(H_j(y))$,
we get $x=G(y)$. Since $\psi\inv$ is injective, $\Ker(G)$ must be
trivial, i.e. $G(y)=ky$ for some $k\in\calu$. However, since $\psi$ is a
field automorphism we must have $k=1$, i.e: $\psi=id$. \qed

\para {\bf Definable and algebraic closure, independence}. Let $(\calu, \Delta)$ be a
differentially closed field. 
If $A\subset \calu$, then $\dcl_\Delta(A)$ and $\acl_\Delta(A)$ denote
the definable and algebraic closure in the sense of the theory $\dcf$.
Then $\dcl_\Delta(A)$ is the smallest differential field
containing $A$, and $\acl_\Delta(A)$ is the field-theoretic algebraic closure
of $\dcl_\Delta(A)$. Independence in $\calu$  is given by independence in the sense
of the theory ACF (of algebraically closed fields) of the algebraic
closures, i.e., $A\dnfo_CB$ iff $\acl_\Delta(CA)$ and $\acl_\Delta(CB)$
are linearly disjoint over $\acl_\Delta(C)$. \\[.1in]
A result we will often use is the following: 

\begin{thm} (Pillay, \cite{P-Fund}, Theorem 4.1 and Corollary 4.2)\vlabel{P-fund} Let $G$ be a
  group definable in $\calu$, which is connected (i.e., with no definable subgroup of finite
  index). Then $G$ embeds definably into $H(\calu)$, for some algebraic
  group $H$.
\end{thm}

\noindent This result is stated for one derivation, but the author remarks that it
generalises easily to several commuting derivations.

\para {\bf The theories {\rm ACFA} and $\dcfa$}\\
The $\call_\si$-theory of difference fields has a
model companion  denoted $\acfa$ (\cite{M}, see also \cite{CH}   and \cite{CHP}). 
Le\'on-S\'anchez showed that the $\call_{\si,\Delta}$-theory of
difference-differential fields admits a model companion, $\dcfa$, and he
gave an explicit axiomatisation of this theory in \cite{LS}. (When
$m=1$, the theory was extensively studied by the third author, in
\cite{Bu2}, see also \cite{Bu1}). \\[0.05in]
The theories $\acfa$ and $\dcfa$ have similar properties, they are model-complete, supersimple and eliminate imaginaries, but they are not complete and  do not eliminate quantifiers.  The completions of both theories are obtained by
describing the isomorphism type of the difference subfield
$\rat^{alg}$. In what follows we will view ACFA as $\dcfa$ with $m=0$,
and we fix a (sufficiently saturated) model $\calu$ of $\dcfa$. Until
subsection 2.3, definable will mean $\call_{\si,\Delta}$-definable in
$\calu$, and a {\em small} $A$ will be one such that $\calu$ is $(|A|+\aleph_0)^+$
saturated.

\para{\bf The fixed field}\\
We let $\fix (\sigma):= \{x \in \calu: \sigma (x)=x\}$ be the fixed
field of $\calu$; by Theorem~\ref{LS1}(8) below, it is a pseudo-finite
field and $\fix (\si^k)$ is the unique extension of $\fix
(\si)$ of degree $k$.

\begin{thm} (\cite{LS}, Propositions 3.1, 3.3 and 3.4). \vlabel{LS1}
 Let $a, b$ be tuples in $\calu$ and let  $A\subseteq \calu$.
We will denote by $\acl(A)$ the model theoretic closure of $A$ in the
$\call_{\si,\Delta}$-structure $\calu$, and by $\calc$ the subfield of
absolute constants of $\calu$. Then:

\begin{enumerate}
\item $\acl(A)$ is the (field-theoretic) algebraic closure of the
  difference-differential field generated by $A$.
\item If $A = \acl(A)$, then the union of the quantifier-free diagramme of $A$ and of the theory $\dcfa$ is a complete theory in the language $\call_{\si,\Delta}(A)$.
	\item $tp(a/A)=tp(b/A)$ if and only if there is an
$\call_{\si,\Delta}(A)$-isomorphism $\acl(Aa)\to \acl(Ab)$ sending $a$
          to $b$.
          \item Every $\call_{\si,\Delta}$-formula $\varphi(x)$ is
            equivalent modulo $\dcfa$ to a disjunction of formulas of
            the form $\exists y\, \psi(x,y)$, where $\psi$ is
            quantifier-free (positive), and such that for every tuples
            $a$ and $b$ (in a difference-differential field of
            characteristic $0$), if $\psi(a,b)$ holds, then $b\in \acl(a)$. 
\item Every completion of $\dcfa$ is supersimple (of SU-rank
  $\omega^{m+1}$). Independence is given by independence (in the sense
  of {\rm ACF}) of algebraically closed sets:\\
$a$ and $b$ are independent over $C$ if and only if the fields
$\acl(Ca)$ and $\acl(Cb)$ are linearly disjoint over $\acl(C)$.
\item Every completion of $\dcfa$ eliminates
  imaginaries.
  \item If $k\geq 1$, and $\calu\models \dcfa$, then  the difference-differential field
$\calu[k]=(\calu,+,\cdot,\Delta,\si^k)$ is also a model of $\dcfa$,  and the algebraic closure of
    $\fix(\si)$ is a model of $\dcf$.
    \item $\fix(\si)$ and $\fix(\si)\cap \calc$ are pseudo-finite fields. 
\end{enumerate}
\end{thm}

\begin{rem} \vlabel{rem-LS1}\begin{enumerate}
    \item[(a)] Item (4) is stated in a slightly different way in
      \cite{LS}. Here we prefer to have our set defined positively, at
      the cost of $y$ consisting of maybe several elements. This gives
      us that every definable subset of $\calu^n$ is the projection of a
      $\si$-$\Delta$-algebraic set $W$ by a projection with finite fibers. 
 \item[(b)]  
By item (5) independence is given by independence (in the sense of ACF) of algebraically closed sets. 
This shows in particular that $\dcfa$ is {\em one-based  over} ACF (see  \cite{BMW}).
\item[(c)] As with ACFA, if $G$ is a definable
  subgroup of $H(\calu)$ for 
  some algebraic group $H$, we  define the {\em prolongations}
  $$p_n:H(\calu)\to H(\calu)\times \si(H(\calu))\times \cdots \times
  \si^n(H(\calu)),\ g\mapsto (g,\si(g),\ldots,\si^n(g)),$$ and let $G_{(n)}$ be the Kolchin closure of
  $p_n(G)$.  Then an element $g\in G$ is a generic of $G$ if and only if  $p_n(g)$ is a generic of the $\Delta$-closed subgroup
  $G_{(n)}$ of $H(\calu)\times \si(H(\calu))\times \cdots \times
  \si^n(H(\calu))$ for
  each $n$. This implies that $G$ and its $\si$-$\Delta$-closure
  have the same SU-rank, and therefore that  $G$  has finite index in its
  $\si$-$\Delta$-closure. 

\item[(d)] Let  $A\subset \calu$ be a small difference-differential
subfield, and let $L$ be a small difference-differential
field extending $A$. Assume that $L\cap
A^{alg}=A$. Then there is an $A$-embedding of $L$ into $\calu$. Indeed,
our assumption implies that $L\otimes_A A^{alg}$ is an integral domain,
whence the difference-differential structure of $L$ and of $A^{alg}$
extend uniquely to the  field $M$ of fractions of $L\otimes_AA^{alg}$.
Because $A^{alg}=\acl(A)$, the conclusion now follows by item (2) and the saturation hypothesis on $\calu$: $qftp(L/A)$ is
realized in $\calu$.
\item[(e)] This  has the following consequence, which we will use:\\
 Let $q$ be a
quantifier-free type over a difference-differential subfield $A$ of $\calu$, and
suppose that $q$ is stationary, i.e., if $a$ realises $q$, then
$A(a)_{\si,\Delta}\cap A^{alg}=A$. Let $f:A\to A'\subset\calu$ be an
isomorphism; then $f(q)$ is realised in $\calu$.

\item[(f)] When $m=0$, all these results appear in \cite{CH}. When
  $m=1$, they appear in \cite{Bu1}, \cite{Bu2}.
  \item[(g)] Item (8) has a more general formulation in Proposition 3.4(iv) of \cite{LS} as it applies to any field of constants of a set $\Delta'\subset\fix(\si)(\calc)\Delta$ of commuting 
    derivations. 

\end{enumerate}

\end{rem}

\noindent
From item (b) above, we can deduce the following:

\begin{lem}\vlabel{gpdef} Let $G$ be a group definable in $\dcfa$. Then there are a
  definable subgroup $G_1$ of finite index in $G$, and a definable
  homomorphism $f: G_1\to H(\calu)$, where $H$ is an algebraic group, and
  the kernel of $f$ is finite. 
\end{lem}


\prf By Remark \ref{rem-LS1}(b), and Theorem 4.9, Remark
4.10 of \cite{BMW}, $G$ has a definable subgroup $G_0$ of finite index,
and there is a definable homomorphism $g:G_0\to H(\calu)$, where $H$ is
an algebraic group and $\Ker(g)$ is finite. \qed

\para {\bf The fixed field}. The fixed field plays an important role in
the study of interactions between definable sets. For instance, any
non-stationary type is non-orthogonal to the fixed field (as proved by
the first author in \cite{Bu2}). The next few results of this subsection are not particularly
difficult, but to our knowledge do not appear in print. 

\begin{defn}
Let $M$ be a $\mathcal{L}$-structure.  A definable subset $D$ of $M^r$ is {\em stably embedded} if
every $M$-definable subset of $D^n$ is definable with parameters from
$D$, for any $n\geq 1$.
\end{defn}

\bigskip\noindent
Let $(\calu, \si, \Delta)$ be a sufficiently saturated model of $\dcfa$.
For $\ell\geq 1$,  we consider the difference-differential field
$F_\ell :=\fix(\si^\ell)$. It is the only Galois extension of $F_1=:F$
of degree
$\ell$, and its Galois group over $\fix(\si)$ is generated by
$\si$. Similarly, if $m\geq 1$, $\fix(\si^{\ell m})$ is the unique Galois
extension of $F$ of degree $m\ell$. The lemma below will be useful in
determining the induced structure.

\begin{lem}\vlabel{F-ell-1} Let $k,\ell\geq 1$ be integers. Then the
  difference-differential field $F_{k\ell }$ is interpretable in the
  difference-differential field $F_k$. Furthermore, given an
  $\call_{\si,\Delta}$-formula $\varphi(x)$ ($x$ an $n$-tuple of variables),
  there is an $\call_{\si,\Delta}(F_k)$-formula $\varphi^*(y)$, $y$
  a $\ell n$-tuple of  variables, and an $F_k$-basis $\calb$ of $F_{k\ell
    }$ such that for any $n$-tuple $c\in F_{k\ell
    }$, if $d$ denotes the coordinates of $c$ with respect to $\calb$,
  then $$F_{k\ell }\models \varphi(c) \iff F_k\models \varphi^*(d).$$
  The formulas $\varphi$ and $\varphi^*$ have the same complexity.
  \end{lem}

\prf We will first do  the case $k=1$. Let $\alpha$ be such that
$F(\alpha)=F_\ell$, and let $f(X)=X^\ell+\sum_{i=0}^{\ell-1}a_iX^i\in F[X]$ be
its minimal polynomial over $F$. If $b_0,\ldots,b_{\ell-1}\in
F$ are such that $\si(\alpha)=\sum_{i=0}^{\ell-1}b_i\alpha^i$, then,
identifying $F_\ell$ with $F\oplus F\alpha\oplus \cdots \oplus
F\alpha^{\ell-1}$, one can define  on $F^\ell$ the structure of the
difference-differential field $F_\ell$: Addition is obvious,
multiplication by $\alpha$ is given by a matrix involving the
coefficients $\bar a$ of $f(X)$, and $\si$ is definable using
$\si(\alpha)=\sum b_i\alpha^i$. As to the derivations, note that for
$1\leq j\leq m$ one has 
$$\delta_j(\alpha)=-\frac{\sum_{i=0}^\ell\delta_j(a_i)\alpha^i}{\sum_{i=1}^\ell
  ia_i\alpha^{i-1}}.$$
This interpretation of $F_\ell$ in $F$ is done quantifier-freely, so
that from a formula $\varphi$ one easily gets the formula $\varphi^*$,
with the same complexity,  when $k=1$. \\[0.05in]
For the general case, let $k>1$, and observe that $F$ is
quantifier-freely definable in the difference-differential field $F_k$
by the formula $\si(x)=x$. If $\alpha$ is such that
$F_{k\ell}=F(\alpha)$, then the procedure given above allows us to
interpret $F_{k\ell}$ inside $F$, by defining an
$\call_{\si,\Delta}$-structure on $F^{k\ell}$. In order to get the full
result, it 
suffices to show that some embedding of the difference-differential field
$F_k$ inside the copy $F^{k\ell}$ of $F_{k\ell}$ is definable: let
$\beta\in F_k$ be such that $F_k=F(\beta)$, write
$\beta=\sum_je_j\alpha^j$, and let $\iota:F_k\to
F^{k\ell}$ be defined by sending $\beta$ to the element with coordinates
$(e_j)$ with respect to the basis $\{1,\alpha,\ldots,\alpha^{k\ell
  -1}\}$. Thus, $\iota$ is the restriction to $F_k$ of the isomorphism
$F_{k\ell}\to F^{k\ell}$. We are therefore able to interpret
quantifier-freely  in $F$ and
also in $F_k$, the $\call_{\si,\Delta}$-structures
$(F_{k\ell},F_k,F,\Delta,\si)$. (This interpretation is uniform in the
parameters $(a,b,e)$.) The result follows easily. \qed

\begin{prop}\vlabel{psf1} Fix $\ell\geq 1$. Then $F_\ell$ is stably embedded, and
 its induced structure is that of the pure difference-differential
  field. If $\ell=1$, it is the pure differential field. 

\end{prop}

\begin{proof}  The first part follows from elimination of imaginaries
  (Prop.~3.3 in \cite{LS}): if $c$ is a code for a definable subset $S$ of
  $F_\ell^n$, then $\si^\ell(c)=c$. So every definable subset of $F_\ell^n$ is
  definable using parameters from $F_\ell$. 

\noindent
  By Proposition \ref{LS1}(4), every $\call_{\si,\Delta}$-formula
  $\varphi(\bar x)$ is
equivalent (modulo $\dcfa$) to a formula of the form $\exists \bar y\,
\psi(\bar x,\bar y)$, where $\psi(\bar x,\bar y)$ is quantifier-free, and whenever $(a,b)$
realises $\psi$, then $b\in \acl(a)$.  Let $r$ be a bound on the degree
of $b$ over the difference-differential field generated by $a$, and
$N$ the least common multiple of all integers $\leq r$. By Lemma
\ref{F-ell-1} applied to $F_{\ell N}$, if one fixes an $F$-basis $\calb$
of $F_{\ell N}$ over $F$, there is an $\call_\Delta(F_\ell)$-formula
$\exists \bar z\,
\psi^*(\bar t,\bar z)$ such that for any $\bar c\in F_{\ell N}$, if $\bar d$
denotes the tuple of coordinates of $\bar c$ with respect to $\calb$,
then $$\calu\models \exists \bar y\, \psi(\bar c,\bar y) \iff 
F\models \exists z\, \psi^*(\bar d,\bar z).$$
This finishes the proof. \qed

\end{proof}

\begin{cor} \vlabel{psf4} 
We consider $F_\ell$ as a difference-differential field
and we denote as $\acl_{F_\ell}(A)$ the algebraic closure in the sense of model theory in $F_\ell$.
If $A\subset F_\ell$, then $\acl_{F_\ell}(A)=\acl(A)\cap F_\ell$, and independence in $F_\ell$ is given by independence (in the sense of
  ACF) of algebraic closures.
\end{cor}

\prf Certainly, in $F_\ell$, $A$ and $B$ independent over $C$ implies
that $\acl_{F_\ell}(CA)$ and $\acl_{F_\ell}(CB)$ are algebraically
independent  over $\acl_{F_\ell}(C)$. Conversely, assume that $A=\acl_{F_\ell}(CA)$ and
$B=\acl_{F_\ell}(CB)$  are 
algebraically independent  over $C=\acl_{F_\ell}(C)$, but that there is some formula
$\psi(a,b,c)\in tp_{F_\ell}(A/B,C)$ ($a\in A$, $b\in B$ and $c\in C$) which forks over $C$. As ${\rm Th}(F_\ell)$ is
supersimple, this means that there is a sequence $(b_i)_{i\in\omega}$ of
independent realisations (in the sense of $F_\ell$) of $tp_{F_\ell}(b/Ca)$
such that $\{\psi(a,b_i,c)\mid i\in\omega\}$ is inconsistent. By the
other direction, we have that, letting $B_i=\acl_{F_\ell}(Cb_i)$, the
difference-differential fields $A$, $(B_i)_{i\in\omega}$ are
algebraically independent over $C$, and therefore so are
$\acl(A)=A^{alg}$, $\acl(B_i)=B_i^{alg}$ over $\acl(C)=C^{alg}$. Observe
also that there is a formula $\psi'(a,b,c)$ such that $$
\calu\models \psi'(a,b,c) \iff F_\ell\models \psi(a,b,c).$$
Thus, we would have that  $\{\psi'(a,b_i,c)\mid i\in\omega\}$ is
inconsistent, which gives the desired contradicition.

\subsection{The results of Cassidy}

\para Before stating Cassidy's results, we need
some definitions and  discussion of the various notions. We work inside a (sufficiently saturated) differentially closed
field $\calu$ of characteristic $0$, and in this subsection definable will mean definable
in the differential field $\calu$. \\[0.05in]

\begin{defn} (See Chapter II in \cite{Ca1}) An (affine) \emph{$\Delta$-algebraic
group}, or {\em differential algebraic group}, is a $\Delta$-closed
  subset of affine space, 
 whose group laws are
 locally given by everywhere defined differential rational maps. 
\end{defn}

\noindent 
This context was extended to the
non-affine setting, see e.g. \cite{Ko2} chapter 1 \S2. A
$\Delta$-algebraic group is then definable in $\calu$, but there are
definable groups which are not $\Delta$-algebraic.

  \begin{defn}
    \begin{enumerate}
\item A (definable) group is {\em linear} if it (definably) embeds into
  some ${\rm GL}_n(\calu)$.  If the algebraic group $G$ has no infinite
  abelian quotient, then $G$ is linear. 
    \item Recall that an algebraic group $G$ is  {\em simple} if it is
      non-abelian, and has no proper connected normal algebraic
      subgroup. It is {\em semi-simple} if it connected, linear, and  has no
      non-trivial connected abelian normal algebraic subgroup. Note
      that a finite 
      center is allowed. A semi-simple algebraic group is
      isogenous to a 
      (finite) product of simple normal algebraic subgroups, called its
      {\em simple components}. See \cite{Hu}, Theorem 27.5. 
      \item 
A $\Delta$-algebraic group $G$ is {\em $\Delta$-simple} if it is non-abelian and has no proper connected normal $\Delta$-closed subgroup. 
Again, a finite center is allowed. 
\item Similarly, a linear $\Delta$-algebraic group $G$ is \emph{$\Delta$-semi-simple}
if  it has no  non-trivial connected normal abelian $\Delta$-closed subgroup. 
\end{enumerate}
\end{defn}

The following results were shown by  Cassidy in \cite{Ca2}:

\begin{thm}\vlabel{thm14} (Cassidy, \cite{Ca2}, Theorem 14). Let $G$ be
  a connected linear $\Delta$-semi-simple $\Delta$-algebraic group. Then $G$ is
   $\Delta$-isomorphic to  a Zariski dense
  $\Delta$-algebraic subgroup of a connected semi-simple algebraic group $H$. \end{thm}

\begin{thm}\vlabel{thm15} (Cassidy, \cite{Ca2}, Theorem 15) Let $G$ be a
Zariski dense connected $\Delta$-closed subgroup
  of a semi-simple algebraic group $A\leq {\rm GL}(n,\calu)$, with
  simple components $A_1,\ldots,A_t$. Then there exist connected
  nontrivial $\Delta$-simple normal $\Delta$-closed subgroups $G_1,\ldots,G_t$
  of $G$ such that
  \begin{enumerate}
  \item If $i\neq j$, then $[G_i,G_j]=1$.
    \item The product morphism $G_1\times \cdots \times G_t\to G$,
      $(g_1,\ldots,g_t)\mapsto g_1 g_2\cdots g_t$,  is  onto, with finite kernel.
      \item $G_i$ is the identity component of $G\cap A_i$, and is
        Zariski dense in $A_i$.
        \item $G$ is $\Delta$-semi-simple.
    \end{enumerate}

\end{thm}

\begin{rem} Note the following consequence: Let
  $G$ be a group definable in $\calu$, with no definable subgroup of
  finite index, and with no definable infinite abelian quotient. Assume that if $A$ is a definable
  infinite abelian subgroup, then $N_G(A)$ has infinite index in
  $G$. Then there is a definable homomorphism $\varphi:G \to H(\calu)$ for
  some semi-simple algebraic group $H$, with $\varphi(G)$ Zariski dense
  in $H$ and $\Ker(\varphi)$ finite. 
  Indeed, Pillay's  Theorem \ref{P-fund} gives the existence of a
  definable embedding 
  $\varphi:G\to H(\calu)$ for some algebraic group $H$, and we may
  assume that $\varphi(G)$ is Zariski dense in $H$. Thus $H$ has no
  infinite definable abelian quotient, since $[G,G]=G$ implies $[H,H]=H$,
so that $H$ is linear. Further, if $A$ is a normal abelian algebraic
subgroup of $H$, then $A(\calu)\cap \varphi(G)$ is normal in $\varphi(G)$, and
therefore must be finite. Replace $H$ by $H/A$, and compose $\varphi$ with
the projection $H\to H/A$. Applying if necessary the same procedure to
$H/A$, we may assume that some quotient of $H$ is semi-simple, and that
the map $\varphi:G\to H(\calu)$ has finite kernel.\\
By 
  quantifier-elimination in $\dcf$, we know that $\varphi(G)$ is a
  $\Delta$-algebraic subgroup of $H(\calu)$. \\
  This is slightly stronger than Cassidy's result, since our group $G$
  is only ``$\call_\Delta$-definable'', i.e., is a Boolean combination
  of $\Delta$-closed sets. 
  \end{rem}

\begin{thm} \vlabel{Cassidy} (Cassidy, \cite{Ca2}, Theorem 19). Let $H$
  be a simple algebraic group defined and split over $\rat$, and $G\leq H(\calu)$ be a
  $\Delta$-algebraic subgroup which is Zariski dense in $H$. Then $G$ is definably isomorphic to $H(L)$,
  where $L$ is the constant field of a set $\Delta'$ of commuting derivations. Furthermore, the isomorphism is
  given by conjugation by an element of $H(\calu)$.  
\end{thm}

\begin{rem}\vlabel{rem-Cass}
  Cassidy's Theorem \ref{Cassidy}  is stated in different terms. Instead of speaking
  of {\em simple algebraic groups, defined and split over $\rat$} in \cite{Ca2}, she
  speaks of {\em simple Chevalley groups}. In fact, many of her results are
  stated in terms of Chevalley groups, but we chose not to do
  that. Recall that any simple algebraic group is isomorphic to one
  which is defined and split over the prime field, $\rat$ in our case. \\
Note also that the field $L$ of Theorem \ref{Cassidy} is
algebraically closed. We will therefore be able to use Fact 
\ref{simple} below.  

\end{rem}

\begin{fact}\vlabel{simple} Let $G$ be a simple algebraic group, let $K$ be an algebraically closed field. Then
  \begin{itemize}
  \item[(a)] The group $G(K)$ has no infinite normal subgroup;
  \item[(b)] The field $K$ is
  definable in the pure group $G(K)$.
  \end{itemize}
 
\end{fact}

\noindent
Both assertions are well-known, but we were not able to find an easy
reference for the first assertion: it follows from the fact that if $g\in G(K)\setminus
Z(G(K))$, then the infinite set $(g^{G(K)}g\inv)$ is Zariski closed and
irreducible,  contains $1$, is closed under conjugation, and therefore generates a Zariski closed normal
subgroup of
$G(K)$, which must equal $G(K)$ since $G$ has no proper infinite normal
algebraic subgroup. The second is also well-known, see for instance 
Theorem~3.2 in \cite{KRT}.

 \subsection{Quantifier-free canonical bases} \vlabel{qfcb}
As $\dcfa$ is
 supersimple there is a notion of canonical basis for complete types
 which is defined as a sort of amalgamation basis, and is not easy to
 describe. In our case, we will focus on an easier concept:  canonical
 bases of quantifier-free types.  They are defined as follows:\\[0.05in] 
   We work in a model $(\calu, \si, \Delta)$ of $\dcfa$.
  Let $a$ be a finite tuple in
   $\calu$, and $K\subset\calu$ a difference-differential field. We
   define the {\em quantifier-free canonical basis} of $tp(a/K)$, denoted
   by $\qfcb(a/K)$, as the smallest difference-differential subfield $k$ of
   $K$ such that $k(a)_{\si,\Delta}$ and $K$ are linearly disjoint over
   $k$. Another way of viewing this field is as the smallest
   difference-differential subfield of $K$ over which the smallest
   $K$-definable $\si$-$\Delta$-closed set  containing $a$ is defined
   (this set is called the
  {\em  $\si$-$\Delta$-locus of $a$ over $K$}). Analogous notions exist
  for $\dcf$ and ACFA. We
   were not able to find explicit statements of the following easy
   consequences of the Noetherianity of the $\si$-$\Delta$-topology, so
   we will indicate a proof.

   \begin{lem} \vlabel{cb1} Let $a,K\subset\calu$ be as above.
     \begin{enumerate}
\item $\qfcb(a/K)$ exists and is unique; it is finitely generated as a
   difference-differential field.

  \item Let $K\subset  M\subset K(a)_{\si,\Delta}$. Then
    $M=K(b)_{\si,\Delta}$ for some finite tuple $b$ in
    $M$. 

       \end{enumerate}

     \end{lem}

\prf (1) Let $n=|a|$, and write  $K\{y\}_\si=\bigcup_{r\in\nat} K[r]$, where  $$K[r]=K[\si^i\delta_1^{i_1}\delta_2^{i_2}\cdots\delta_m^{i_m}y_j\mid
   1\leq j\leq n, |i|+\sum_j i_j\leq r].$$
Then each $K[r]$ is finitely generated over $K$ as a ring, and is
therefore Noetherian. For each $r$, consider the ideal $I[r]=\{f\in K[r]\mid f(a)=0\}$, and
the corresponding $\si$-$\Delta$-closed subset $X[r]$ of $\calu^n$
defined by $I[r]$. Then the sets $X[r]$ form a decreasing sequence of
$\si$-$\Delta$-closed subsets of $\calu^n$, which stabilises for some
$r$, which we now fix. Note that the ideal $I[r]$ is a prime ideal (of
the polynomial ring $K[r]$), and as such has a smallest field of
definition, say $k_0$, and that   $k_0$ is finitely generated as a
field, and is unique. We now let $k$ be the difference-differential
field generated by $k_0$. \\[0.05in]
{\bf Claim 1}.
$k(a)_{\si,\Delta}$ and $K$ are linearly disjoint over $k$.

\begin{proof}
This follows from the fact that $X[s]=X[r]$ for every $s\geq r$.
\end{proof}

\noindent
(2) Consider $B:=\qfcb(a/M)$. By (1), $B$ is finitely generated as a
difference-differential field. \\[0.05in]
{\bf Claim 2}.  $KB=M$.

\begin{proof}
Indeed, by definition, $B(a)_{\si,\Delta}$ and $M$ are linearly disjoint
over $B$. Hence, $KB(a)_{\si,\Delta}$ and $M$ are linearly disjoint over
$KB$. But this is only possible if $KB=M$.
\end{proof}

\begin{rem}\vlabel{canonicalbase}
Given  fields $K\subset L$ (of characteristic $0$), 
    the field $L$ is a regular extension of $L_0:=K^{alg}\cap L$. So, if
    $L=K(a)_{\si,\Delta}$ for some (maybe infinite) tuple $a$, then
    $B:=\qfcb(a/K^{alg})$  is contained in $L_0$, and we have
    $KB=L_0$: the inclusion is clear, and the reverse
    inclusion follows by noting that  the linear disjointness of
    $B(a)_{\si,\Delta}$ and 
    $K^{alg}$ over $B$ implies the linearly disjointness of
    $KB(a)_{\si,\Delta}$ and $K^{alg}$ over $KB$, whence $KB$ must
    contain $L_0$. \\[0.05in]
In positive characteristic $p$, if $L$ is a separable extension of $K$,
then $L$ and $K^s$ are linearly disjoint over their intersection $L_0$:
this is because $K^s/K$ is Galois.  {We will use this remark later}   

\end{rem}
 
\subsection{Some well-known group-theoretic results}

\begin{lem}\vlabel{subdirect}
 Let $G_1$ and $G_2$ be groups, and $H\leq G_1\times G_2$ a proper subgroup,
 such that $H$ projects onto $G_1$ and onto $G_2$ (via the two natural
 projections $G_1\times G_2\to G_i$). Let $S_1$ be defined by $H\cap
 (G_1\times (1))=S_1\times (1)$
 and $S_2$ by $H\cap  ((1)\times G_2)=(1)\times S_2$. Then $S_i$ is a
 normal subgroup of $G_i$, and $H/(S_1\times S_2)$ is the graph of a group
 isomorphism $f:G_1/S_1\to G_2/S_2$.\\
 Furthermore, if the $G_i$ and $H$ are definable, so are the $S_i$ and
 the isomorphism $f$. 

\end{lem}

\prf Let $g\in G_1$; because $H$ projects onto $G_1$, there is some
$g'\in G_2$ such that $(g,g')\in H$; thus $(g,g')\inv (S_1\times
(1))(g,g')=g\inv S_1g\times (1)$ is contained in $H$, which shows
$S_1\unlhd G_1$. Similarly, $S_2\unlhd G_2$.\\
For the second assertion, we may assume that $S_1=S_2=(1)$. I.e, given
$g\in G_1$ there is a unique $g'\in G_2$ such that $(g,g')\in H$, and
given $g'\in G_2$ there is a unique $g\in G_1$ such that $(g,g')\in H$:
this exactly says that $H$ is the graph of a bijection between $G_1$ and
$G_2$, and a moment's thought gives that it is an isomorphism.\\
The last statement is obvious.  \qed 

\begin{cor}\vlabel{subdirect-ss} Let $H\leq G\times G_2$ be 
  as above, with $G$ and $H$ definable, centerless, and $G_2=\prod_{i=1}^sG'_i$,
  and where $G$ and  each $G'_i$ is definable and definably simple centerless. Then for some
  index $i$, the image of $H$ in $G\times G'_i$ via the natural
  projection $G\times G_2\to G\times G'_i$, is the graph of a
  definable isomorphism $f$ between $G$ and $G'_i$. \\
  If moreover $G=H(L)$ and
  the $G'_i=H_i(L_i)$, where $H$ and the $H_i$ are simple algebraic
  groups defined over $\rat$,
  and the fields $L$ and $L_i$ are fields of constants of the
  differentially closed field  $\calu$, then $f$ is the
  restriction of an algebraic isomorphism $H\to H_i$ and $L=L_i$. 
 
\end{cor}

\prf In the notation of Lemma \ref{subdirect}, $S_1=(1)$, and $S_2$ is a
definable 
normal subgroup of $\prod G'_j$, hence is a product of some of the
factors,  and in fact of $s-1$ of them:  $H/(S_1\times S_2)$ is the graph
of  a group isomorphism $f$ between
the definably simple group $G$ and $G_1/S_2$, so $G_1/S_2$ must be (isomorphic via
the natural projection to) one of the 
 $G'_i$. 

The moreover part follows from  results of Borel-Tits (see Theorem A in
\cite{Bo-Ti}, or 2.7, 2.8 in \cite{St72}, or Theorem~4.17 in \cite{Po})
which describe abstract isomorphisms between simple algebraic groups:
there are an algebraic isomorphism $\varphi:H\to H_i$, and field
isomorphism $\psi:L\to L_i$, such that $f=\bar\psi\circ
\varphi$, where $\bar\psi$ is the isomorphism $H_i(L)\to H_i(L_i)$
induced by $\psi$. Since $f$ and $\varphi$ are definable, so is
$\bar\psi$. By  Fact \ref{simple}(b), $\psi$ is also definable. A result of S. Suer (Theorem 3.6 in \cite{S-JSL2012})
then tells us that $L=L_i$, and using Theorem \ref{autL}, we obtain that
$\psi$ is the identity. Hence $f=\varphi$. \qed

\begin{defn} \vlabel{abuse} Let $G_i\leq H_i(\calu)$ be Zariski dense subgroups of
the algebraic groups  $H_i$ for $i=1,2$, and let $f:G_1\to G_2$ be an isomorphism. By abuse
  of language, we will
  say that $f$ is an {\em algebraic isomorphism} if it is the
  restriction to $G_1$ of an algebraic isomorphism $H_1\to H_2$.
  \end{defn}

\section{The isogeny result}

\noindent
We  work in a  sufficiently saturated model $(\calu, \Delta, \si)$ of
$\dcfa$. We will often work in its reduct to $\call_\Delta$. Unless
otherwise mentioned, definable will mean
$\call_{\si,\Delta}$-definable in $\calu$ with parameters.

\begin{defn}\vlabel{definably-quasi-simple} Let $G_1$, $G_2$, $G$ be  definable
  groups.
  \begin{enumerate}
    
\item Recall that $G_1$ and $G_2$ are  {\em definably isogenous} if there are
  definable subgroups $H_i$ of $G_i$ of finite index, 
  finite normal subgroups $S_i$ of $H_i$ for $i=1,2$, and a definable
  isomorphism 
  $H_1/S_1\to H_2/S_2$.
  
  \item 
    We say that $G$ is
  \emph{definably quasi-simple} if $G$ has no definable abelian subgroup
  of finite index, 
      and if  whenever $H$ is a definable
  infinite subgroup of $G$ of infinite index, then its normaliser $N_G(H)$ has infinite
  index in $G$.
  \item 
    We say that $G$ is {\em definably quasi-semi-simple} if 
        whenever $H$ is a
  definable infinite abelian subgroup of $G$, then its normaliser
  $N_G(H)$ has infinite index in $G$, and if $G$ has no definable
    normal subgroup $N$ such that $G/N$ is definably isogenous to an
    infinite abelian group. 
\end{enumerate}
\end{defn}

\begin{remark}\vlabel{rem-defss} Conditions (2) and (3) will ensure that
  if there is a definable homomorphism $f:G\to H(\calu)$ for some
  algebraic group $\calu$, and with $\Ker(f)$ finite, then the Zariski
  closure of $f(G)$ will be a linear algebraic group. \\
  In our context (of a supersimple theory),  a definable group will in general have infinitely many
  definable subgroups of finite index, so it will not have a smallest
  one. We devised these properties to take care of that problem, and we
  will show below that they are preserved by isogenies. We first need a
  lemma:\end{remark} 
  \begin{lem} \vlabel{abelian} Assume that $G$ is an infinite definable 
    group. 
\begin{enumerate}
\item Assume that $G$ is  abelian, and let $0\neq
  n\in\nat$. Then $[G:G^n]$ is finite.
\item Assume that $G$ is definably isogenous to an abelian group. Then
  $G$ has an infinite abelian subgroup $A$ of finite index.    

  \end{enumerate}
  \end{lem}

\prf (1) By Lemma \ref{gpdef}, we know that there are a subgroup $A$
of finite index in $G$,   and a definable $f:A\to H(\calu)$ for some
algebraic group $H$, with $\Ker(f)$ finite.  Then it
suffices to show the conclusion for $f(A)$: because $\Ker(f)$ is finite,
$A^n$ has index at most $|\Ker(f)|$ in  $f\inv(f(A))^n$.

Hence we may assume that $A$
 is Zariski dense in
 $H$. Then $H$ is abelian as well,  and we may assume it is connected.

 \smallskip\noindent
     {\bf Claim}. The $n$-torsion subgroup $Z$ of $H(\calu)$ is finite.\\
     The proof is by induction on $\dim(H)$. Assume first that $H$ has
     no proper algebraic subgroup. Then $H$ is one of $\ga$, $\gm$ or a
     simple abelian variety. In all three cases, $Z$ is finite (characteristic $0$ implies that $\ga(\calu)$ is torsion-free). 

  Assume now that     $N$
  is a minimal (proper, infinite) algebraic subgroup of $H$; then by the first
  case, $Z\cap N(\calu)$ is finite. The group $ZN(\calu)/N(\calu)$ is contained in the
  $n$-torsion subgroup of $(H/N)(\calu)$, which is finite by induction
  hypothesis. This proves the claim. 

\medskip\noindent
The claim implies that $A$ and
$A^n$ have the same SU-rank, and therefore that $A^n$ has finite index
in $A$. As $[G:A]$ is finite and $A^n\leq G^n$, we obtain the result. \\[0.05in]
(2)  Let $A$ be a definable subgroup of $G$ of finite
index, and $Z$ a normal finite (of size $n$) subgroup of $A$ such that $A/Z$ is
abelian and infinite.\\ 
Going to a definable subgroup of $A$ of finite index ($C_A(Z)$, which
has index bounded by the size of the largest abelian subgroup of ${\rm Sym}(n)$), we may
assume that  $Z$ is central in $A$.  Thus the map
$$A\times A\to Z,\ (g,h)\mapsto [g,h]$$
is a bilinear map. Then $A^n$ is a definable
subgroup of $A$, which is abelian since every element of $A^n$ is of
the form $g^n$ for some $g\in A$. By (1), $(A/Z)^n$ has finite index in
$A/Z$, and this implies that $A^n$ has finite index in $A$, 
and therefore also in $G$. \qed

  \begin{lem}\label{lem-defss} Let $G$ be a definable group, $G_0$ a definable
    subgroup of $G$ of finite index, and $Z$ a finite normal subgroup of $G$.
    \begin{enumerate}
 \item Assume that $G/Z$ has a definable abelian subgroup $A'$, such that
   $N_{G/Z}(A')$ has finite index in $G/Z$. Then $G$ has a definable
   abelian 
   subgroup $A$, with $N_G(A)$ of finite index. Furthermore, $[A':ZA]$
   is finite.        
\item $G$ is definably quasi-simple if and only if $G_0$ is definably
  quasi-simple.
  \item $G$ is definably quasi-simple if and only if $G/Z$ is definably
    quasi-simple. 
 \item The  assertions of items (2) and (3) hold with ``quasi-semi-simple'' in
   place of  ``quasi-simple''.     
      \end{enumerate}
  \end{lem}

  \prf 
  (1) Going to a definable sugroup of $G$ of finite index, we
  may assume that $Z$ is central in $G$. Let $A\leq G$ be the 
  subgroup of $G$ which contains $Z$ and
  projects onto $A'$. Then $A$  is definable, and by Lemma \ref{abelian}(2), it has a definable
  abelian subgroup $A_0$ of finite index. The
  proof of Lemma \ref{abelian}(2)  shows that in fact we may take
  $A_0=A^{|Z|}$. As $N_{G/Z}(A')$ has finite index in
  $G/Z$, so does $N_G(A)$ in $G$, because $Z$ is central in $G$, and so
  does $N_G(A_0)$ since $A_0$ is a characteristic subgroup of $A$.\\[0.05in] 
 (2)  Suppose $G_0$ is definably quasi-simple,  let $H$ be an infinite definable subgroup of $G$ of infinite
  index, and assume that $N_G(H)$ has finite index in $G$. Then
  $N_G(H)\cap N_G(G_0)$ has finite index in $G$,  is contained in $N_{G}(H\cap
  G_0)$, and therefore
  $N_G(H\cap G_0)\cap G_0=N_{G_0}(H\cap G_0)$ has finite index in $G_0$.
  \\
  However $[H:H\cap G_0]$ finite implies $H\cap G_0$ infinite; as $G_0$ is definably quasi-simple,
  $N_{G_0}(H\cap G_0)$ has  infinite index in $G_0$, therefore in $G$, and
  we get the desired contradiction. That $G$ has no definable abelian subgroup
  of finite index is clear. \\ 
  For the other direction, assume that $G$ is definably
  quasi-simple. Let 
  $H$
  be an infinite definable  subgroup of $G_0$ of infinite index in
 $G_0$, and suppose that 
  $N_{G_0}(H)$ has finite index in $G_0$; then $N_G(H)$, which contains
  $N_{G_0}(H)$,  has finite index
  in $G$, which gives us the desired contradiction. Clearly if $A$ is a
  definable abelian subgroup of $G_0$ of finite index in $G_0$, then $A$
  has finite index in $G$. \\[0.05in]
(3) By (2), going to a definable sugroup of $G$ of finite index, we
  may assume that $Z$ is central in $G$. We will first deal with the
  condition on definable abelian subgroups. If $A$ is a definable
  abelian subgroup of finite index in $G$, then so is $AZ/Z$ in
  $G/Z$. If $A$ is a definable abelian subgroup of $G/Z$ of finite index, use
  Lemma \ref{abelian} to obtain a definable abelian subgroup $B$ of
  finite index in $G$. \\[0.05in]
  Assume $G/Z$ is definably
  quasi-simple, and let $H$ be an infinite definable
  subgroup of $G$ of infinite index.  Then $HZ/Z$ is infinite and has
  infinite index in $G/Z$, so its normalizer $N$ has infinite index in
  $G/Z$, and if $N'\geq Z$ is such that $N'/Z=N$, then $N'$ has
  infinite index in $G$ and normalizes $HZ$. As $Z$ is normal
  in $G$, we have $N_G(H)/Z \leq N$, which shows that $[G:N_G(H)]$ is
  infinite. The other direction is immediate
  because $Z$ is normal. \\[0.05in]
  (4) To prove the equivalence of the first condition of definably quasi-semi-simple: reason as in (2) and (3), noting that if $Z$ is the center of $G$, 
   $A'\leq G/Z$ is abelian definable  
  and $N_{G/H}(A')$ has finite
  index in $G$, and if $A$ lifts $A'$, then $N_G(A)$ has also finite
  index in $G$, and normalises the characteristic abelian subgroup
  $A^{|Z|}$ of $A$ (see the proof of Lemma \ref{abelian}). \\
  To prove the equivalence of the second condition: let $G_0$ be a definable subgroup of $G$ of
  finite index, and assume that it has a normal subgroup $N$ such that
  $G_0/N$ is isogenous to an abelian group (and therefore has a
  definable abelian subgroup of finite index, by Lemma \ref{abelian}). Without loss of generality, $G_0$
  is normal in $G$, so that if $g\in G$, then $G_0/N^g$ is virtually
  abelian. As $N$ has finitely many conjugates, the map $G_0\mapsto
  \prod_{g\in G/G_0}G_0/N^g$ has kernel $M:=\bigcap N^g$, so that
  $G_0/M$ embeds into the virtually abelian group $\prod_{g\in
    G/G_0}G_0/N^g$. The other direction is clear: if $G/N$ is virtually
  abelian, then so is $G_0/N\cap G_0$. \\
  Finally, let $Z$ be a finite normal subgroup of $G$, which we may
  assume to be central. A normal subgroup $N$ of $G/Z$ with $(G/Z)/N$
  virtually abelian, pulls back to a normal subgroup $N'$ of $G$ with
  $G/N'\simeq (G/Z)/N$; and if $N$ is a normal subgroup of $G$ with
  $G/N$ virtually abelian, then so is $GZ/NZ$. \qed

\begin{prop} \label{thm1} Let $G$ be a group definable in $\calu$, and
  assume that $G$ is definably quasi-simple (resp. definably
  quasi-semi-simple). Then there are a definable subgroup $G_0$ of finite index in $G$,
  a simple (resp. semi-simple) algebraic 
  group $H$ defined and split over $\rat$, 
  and a definable homomorphism $\phi:G_0\to H(\calu)$, with finite
  kernel and Zariski dense image. Moreover, if $\bar
  G$ is the connected component of the Kolchin closure of $\phi(G_0)$, then
  $\bar G$ is $\Delta$-simple (resp. $\Delta$-semi-simple).
\end{prop}

\prf By Lemma \ref{gpdef}, there is a definable subgroup $G_0$ of $G$ of
finite index, and a definable homomorphism $\phi:G_0\to H(\calu)$ where
$H$ is a connected algebraic group, $\phi(G_0)$ is Zariski dense in $H$,
and $\Ker(\phi)$ is finite. Note that, as $G_0$ has no infinite
definable quotient which is definably isogenous to an abelian group, the
connected 
algebraic group $H$ has no abelian quotient, and therefore $H$ is
linear. We will now show that we can assume that $H$ is simple
(resp. semi-simple). \\[0.05in]
Assume first that $G$ is definably quasi-simple, and let $N$ be a maximal normal proper algebraic subgroup of
$H$. Then  $N(\calu)\cap \phi(G_0)$ is
finite, and composing $\phi$ with the natural projection $H(\calu)\to
(H/N)(\calu)$, replacing $H$ by $H/N$, we obtain that $H$ can be chosen
to be simple. \\[0.05in]
Similarly, if $G$ is definably quasi-semi-simple, and because $H$ is linear, there is an algebraic connected normal subgroup  $N$
of $H$, maximal among the connected normal solvable algebraic subgroups
of $H$. Then $H/N$ has no proper connected normal abelian algebraic
subgroup, and therefore is semi-simple. An easy induction on the class
of solvability of $N$ shows that $N(\calu)\cap \phi(G_0)$ must be
finite. Again, we may replace $H$ by $H/N$ to obtain that $H$ is
semi-simple. \\[0.05in]
Let $\bar G$ be the Kolchin closure of $\phi(G_0)$ in $H$; replacing
$G_0$ by a subgroup of finite index, we may assume that $\bar G$ is
connected for the Kolchin topology. Then $\bar G$ is Zariski dense in $H$, and by Theorem
\ref{thm15}, $\bar G$ is $\Delta$-semi-simple. Further, if $H$ is
simple, then $\bar G$ is $\Delta$-simple. That $H$ can be taken defined
and split over $\rat$ is because every semi-simple algebraic group is
isomorphic to one such. \qed

\section{Definable subgroups of semi-simple algebraic 
  groups} \label{SDefGroups}

In this section we give a description of Zariski dense definable subgroups of simple and semi-simple algebraic groups. 
We  work in a  sufficiently saturated model $(\calu, \Delta, \si)$ of
$\dcfa$.  Unless
otherwise mentioned, definable will mean
$\call_{\si,\Delta}$-definable with parameters.

\begin{thm} \label{prop1}  Let $H$ be a simple algebraic group defined
  and split 
  over $\rat$, and $G$ a definable subgroup
  of $H(\calu)$ which is Zariski dense in $H$.\\
  Then $G$ has a definable subgroup $G_0$ of finite index, the
        Kolchin closure of which   is conjugate to $H(L)$ (say by an
        element $g\in H(\calu)$), where $L$ 
         is a field of constants in $\calu$. Furthermore,
         either $G_0^g=H(L)$, or $G_0^g\subseteq H(\fix(\si^\ell)(L))$ for some
         integer $\ell\geq 1$. 
In the latter case, if $H$ is centerless, we are able to describe
precisely the subgroup $G_0^g$ as $\{h\in H(L)\mid \si^n(h)=\varphi(h)\}$
for some $n$ and algebraic automorphism
$\varphi$ of $H(L)$.

\end{thm}

\prf  Replacing $G$ by a subgroup of finite index, we may
assume that the  Kolchin closure $\bar G$ of $G$ is connected. Then
$\bar G$ is also Zariski dense in $H$, and by Theorem \ref{Cassidy},
$\bar G$ is conjugate to $H(L)$, for some field of constants $L\leq
\calu$. We may therefore assume that $G\leq H(L)$.\\[0.05in]
The strategy is the same as in the proof of Proposition 7.10 in
\cite{CHP}. Going to the $\si$-closure of $G$ within $H(L)$, and then to a subgroup
of finite index, we may assume that $G$ is quantifier-free definable,
and that it is connected for the $\si$-$\Delta$-topology. If $G=H(L)$, then we
are done. Assume  that $G\neq
H(L)$.  We will first do the case
when $H$ is centerless.\\[0.05in]
In the notation of Remark \ref{rem-LS1}(c), let $n$ be the smallest integer such
that the $\Delta$-algebraic group  $G_{(n)}$ is not equal to $H(L)\times \si(H(L))\times
\cdots\times \si^n (H(L))$ (Here we use that  $G$ is quantifier-free
definable and connected for the $\si$-$\Delta$-topology). If $\pi$ is the projection on the last factor
$\si^n(H(L))$, then $\pi(G_{(n)})=\si^n(H(L))$: this is because if $h$ is a
generic of $G$ for the $\si$-$\Delta$-topology,  then $h$, $\si^n(h)$ are
generics of $H(L)$,  $\si^n(H(L))$ respectively for the Kolchin
topology. As $G_{(n)}$ projects onto $G_{(n-1)}$, it follows that
$G_{(n)}\leq G_{(n-1)}\times \si^n(H(L))$ satisfy the hypotheses of
Lemma \ref{subdirect} and its Corollary \ref{subdirect-ss}. \\[0.05in]
Using Corollary \ref{subdirect-ss} and the fact that $H$ is defined over
$\rat$, it follows that  $\si^i(L)=L$ for some $i$, and that 
for some algebraic automorphism $\varphi$ of $H(L)$, the group 
$G_{(n)}/\prod_{0\leq j\leq n, j\neq i}\si^j(H(L))$ defines the graph of the
restriction of $\varphi$ to $H(L)$. By minimality of $n$,
we have $i=n$, $\si^n(L)=L$. \\
Thus, the group $G$ will be defined by the equation
$\si^n(g)=\varphi(g)$ within $H(L)$ (the left to right inclusion is
clear; equality comes from the fact that both  $G$ and $\{h\in H(L)\mid \si^n(h) = \varphi(h)\}$ are quantifier-free definable and connected for the
$\si$-$\Delta$-topology). 

\medskip\noindent
We now show, still with $H$ centerless, that if $G\neq H(L)$ is as
above, then $G\leq H(\fix(\si^\ell)(L))$ for
some $\ell$. 
By Proposition~14.9 of \cite{Bo}, the group
${\rm Inn}(H)$ of
inner automorphisms of $H(L)$ has finite index in the group 
$\Aut(H)$ of
algebraic automorphisms of $H(L)$. Moreover $\si^n$ induces a
permutation of $\Aut(H)/{\rm Inn}(H)$, and hence there are some $r\in\nat^*$ and
$h\in H(L)$ such
that $$\si^{n(r-1)}(\varphi)\circ\si^{n(r-2)}(\varphi)\circ \cdots \circ\varphi=\lambda_h,$$ where
  $\lambda_h$ is conjugation by $h$. 
 I.e., our group $G$ is contained in the subgroup $G'$ of $H(L)$ defined by
$\{g\in H(L)\mid \si^{nr}(g)=\lambda_h(g)\}$.\\
By the existential closedness of $\calu$, there is some $u\in H(L)$ such that $\si^{nr}(u)=h\inv
u$. So, if $g\in G'$, then
\begin{align*}
  \si^{nr}(u\inv gu)&=\si^{nr}(u\inv)\lambda_h(g)\si^{nr}(u)\\
  &= u\inv h (h\inv gh)(h\inv u)\\
  &=u\inv gu.
  \end{align*}
I.e., $u\inv G'u\leq H(\fix(\si^{nr})\cap L))$. \\[0.1in]
This does the case when $H$ is centerless. 
Assume now that  the center $Z$ of $H$ is non-trivial. By the first part we know
that there are $u\in H(\calu)$ and $\ell\geq 1$ such that $(u\inv
GZu)/Z\leq 
(H/Z)(\fix(\si^{\ell}(L)))$. Since $Z$ is finite and characteristic,
there is some $s\in\nat$ such that for all $a\in Z$, we have
$\prod_{i=0}^{s-1}\si^i(a)=1$. If $g\in u\inv
Gu$, then $\si^\ell(g)g\inv\in Z$; hence $\si^{\ell s}(g)g\inv=1$, and
$u\inv G u\leq H(\fix(\si^{\ell s}))$. \qed

\begin{cor} \label{MainTheorem}
Let $G$ be an infinite group  definable in 
$\calu$, and suppose that $G$ is definably quasi-simple. 
Then there are a simple algebraic group $H$ defined and split over
$\rat$, a definable subgroup $G_0$ of $G$
of finite index, and a definable group homomorphism $\phi:G_0\to H(\calu)$, with the following properties:
\begin{enumerate}
	\item $\Ker(\phi)$ is finite.
	\item  The Kolchin closure of $\phi(G_0)$ is $H(L)$ for some
          field $L$ of constants 
 in $\calu$.
\item Either $\phi(G_0)=H(L)$, or for some integer $\ell$,
$\phi(G_0)$ is a subgroup of $H(\fix(\si^\ell)\cap L)$. 
\end{enumerate}

\end{cor}

\begin{proof}
By Proposition \ref{thm1} we can reduce to the case where $G$ is a
 definable subgroup of a simple algebraic group $H$.
Then apply Theorem  \ref{prop1} to conclude.    
\end{proof}

\begin{lem}\label{lem2} Let $H$ be a simple algebraic group, defined
  and split over $\rat$, let $L\leq \calu$ be a field of constants, and
  let $\varphi$ be an algebraic     
  automorphism of $H$. Let $\ell\geq 1$, and consider the subgroup
  $G\leq H(L)$ defined by $\si^\ell(g)=\varphi(g)$. Then $G$ is
  definably quasi-simple.
  \end{lem}

 \prf By Lemma \ref{lem-defss}, we may assume that $Z(H)=(1)$. Note
 that, as $\varphi$ is defined over $\rat^{alg}$ and therefore over $L$,
 the existential closedness of $\calu$ gives that $G$ is Kolchin dense
 in $H(L)$, and in particular Zariski dense in $H$, so that no infinite abelian
 subgroup of $G$ can have finite index in $G$ (since its Zariski closure 
 would be abelian). Note also that as
 $H(L)$ is connected (for the Kolchin topology), the group $G$ is
 connected for the $\si$-$\Delta$-topology. \\[0.05in]
 Let $U$ be an infinite definable
  subgroup of $G$ of infinite index, and assume by way of contradiction
  that its normalizer $N$ has finite index in $G$. \\
Consider $p_{\ell}$ as defined in Remark \ref{rem-LS1}(c), and $U_{(\ell)}\leq
G_{(\ell)}$. Then $$U_{(\ell)}\unlhd 
N_{(\ell)}=G_{(\ell)}.$$ Indeed, we know that $p_\ell(U)$ and $p_\ell(N)$ are Kolchin
dense in $U_{(\ell)}$ and $N_{(\ell)}$ respectively, with $p_\ell(U)\unlhd p_\ell(N)$; hence any element
of $p_\ell(N)$ normalizes the Kolchin closure $U_{(\ell)}$ of $p_\ell(U)$, and since
$N_{G_{(\ell)}}(U_{(\ell)})$ is Kolchin closed we get $U_{(\ell)}\unlhd 
N_{(\ell)}$; that $N_{(\ell)}=G_{(\ell)}$ is  because
$[G:N]$ is finite and the  group $G_{(\ell)}$ is
connected for the Kolchin topology. In particular, $U_{(0)}\unlhd G_{(0)}=H(L)$,
and as the  group
  $H(L)$ is simple (by Fact \ref{simple}(a)), the Kolchin closure of $U$ must be
  $H(L)$. \\
Moreover,  as every generic of $U$ is a generic of its
$\si$-$\Delta$-closure $\tilde U$, it  follows
that $G$ normalizes $\tilde U$. So, we may replace $U$ by $\tilde U$; then $G$ also normalises the connected
component of $U$ (for the $\si$-$\Delta$-topology), and so we may
assume that $U$ is $\si$-$\Delta$-closed and connected. By Theorem \ref{prop1}, for some $r\leq \ell$
  and algebraic automorphism $\psi$ of $H(L)$, the group $U$ is defined
  within $H(L)$ by
  the equation $\si^r(g)=\psi(g)$. We will show that this is impossible
  unless $r=\ell$ (and $\psi=\varphi$). Indeed, suppose that $r<\ell$,
 take a generic $(u,g)$ of $U\times G$ over all parameters necessary to
 define $G$, $\varphi$, $U$ and $\psi$. 
Consider now
  $(u,\si^r(u))$, and $(g,\si^r(g))$. The elements $u$, $g$ and
  $\si^r(g)$ are independent generics of the algebraic group $H$. Then $$\si^r(g\inv ug)=
\si^r(g\inv)\psi(u)\si^r(g), \ \ \psi(g\inv
ug)=\psi(g\inv)\psi(u)\psi(g).$$ However, since
  $u\in U\unlhd G$, we have $\si(g\inv ug)=\psi(g\inv u g)$, i.e., $\si^r(g)\psi(g\inv)\in
  C_H(\psi(u))$. As $\psi$ is an automorphism of $H$,  the elements $\si^r(g)$, $\psi(g)$ and $\psi(u)$
  are independent generics of $H$;  this gives us the desired
  contradiction, as $\si^r(g)\psi(g)\inv$ and $\psi(u)$ are independent
  generics of the non-abelian algebraic group $H$. \qed

  \para The semi-simple case needs some additional lemmas. Indeed, Zariski
denseness, or even Kolchin denseness,  and the previous results do not suffice to give a complete
description. Here is a typical example: Let $H$ be a simple algebraic
group defined and split over $\rat$, and
consider the subgroup $G$ of $H(\calu)^2$ defined by $$G=\{(g_1,g_2)\in
H(\calu)^2\mid \si(g_1)=g_2\}.$$
Then $G$ is Kolchin dense in $H(\calu)^2$, however $G$ is isomorphic to
$H(\calu)$, via the projection on the first factor. We will now
prove several lemmas which will allow us to bypass this difficulty.

\begin{lem}\label{lem3} Let $G_1,\ldots,G_t$ be centerless $\Delta$-simple
$\Delta$-algebraic groups, with $G_i$ Zariski dense in some simple algebraic
group $H_i$, and $G\leq G_1\times \cdots\times G_t$ a
$\Delta$-algebraic connected (for the Kolchin topology) subgroup, which projects via the natural
projections onto each $G_i$. Then there are a set $\Psi\subset
\{1,\ldots,t\}^2$, and algebraic isomorphisms $\psi_{i,j}:H_i\to H_j$
whenever $(i,j)\in \Psi$, such that
$$G=\{(g_1,\ldots,g_t)\in \prod_{i=1}^tG_i\mid g_j=\psi_{i,j}(g_i),
(i,j)\in \Psi\}.$$
Moreover, if $(i,j), (j,k)\in \Psi$ with $k\neq i$, then $(j,i), (i,k)\in \Psi$,
$\psi_{j,i}=\psi_{i,j}\inv$, and 
$\psi_{i,k}=\psi_{j,k}\psi_{i,j}$.\\
Furthermore, $G$ has no subgroup of finite index.
\end{lem}

\prf By Corollary \ref{subdirect-ss}, if $G$ projects onto
$\prod_{i=2}^t G_i$ and onto $G_1$, but $G\neq \prod_{i=1}^t G_i$, then there is some
index $i\geq 2$, and an 
isomorphism $\psi_{1,i}:G_1\to G_i$ such that $$G=\{(g_1,\ldots,g_t)\in \prod
G_i\mid g_i=\psi_{1,i}(g_1)\}.$$

\noindent 
We define $\Psi$ be the set of pairs $(i,j)\in \{1,\ldots,t\}^2$ such
that the image $G_{i,j}$  of $G$ under the natural projection $\prod_{\ell=1}^t H_\ell\to
H_i\times H_j$ is a proper subgroup of 
$G_i\times G_j$. By  Corollary \ref{subdirect-ss} again, if $(i,j)\in \Psi$, then $G_{i,j}$ is
the graph of an algebraic isomorphism $G_i\to G_j$ (i.e., the
restriction to $G_i$ of an algebraic isomorphism $H_i\to H_j$). Then the set $(\Psi, \psi_{i,j})$
satisfies 
the moreover part of the conclusion,  and we have 
$$G\leq \{(g_1,\ldots,g_t)\in \prod_{i=1}^tG_i\mid g_j=\psi_{i,j}(g_i),
(i,j)\in \Psi\}.$$
To prove equality, we let $T\subset \{1,\ldots,t\}$ be maximal such that
whenever $i,j\in T$, then $(i,j)\notin\Psi$; then the natural projection $\prod_{\ell=1}^t
H_\ell\to \prod_{\ell\in T}H_\ell$ defines an injection on $G$, and
sends $G$ to a subgroup $G'$ of $\prod_{\ell\in T}G_\ell$, with the property
that whenever $k\neq \ell\in T$, then $G'$ projects onto $G_k\times
G_\ell$. By the first case and an easy induction, this implies that $G'=
\prod_{\ell\in T}G_\ell$.
The last assertion follows from it being true  for each
$G_i$ by Theorem \ref{Cassidy} and Fact \ref{simple}.  \qed 

\begin{lem} \label{lem4} Let $H_1,\ldots,H_r$ be simple centerless
  algebraic groups defined and split over $\rat$, $L_1,\ldots,L_r$
  $\Delta$-closed  subfields
  of $\calu$, and $G\leq \prod_{i=1}^rH_i(L_i)$  a Kolchin dense 
  quantifier-free  definable subgroup, which is connected
  for the $\si$-$\Delta$-topology. Let $\tilde G_i\leq H_i(L_i)$ be the
  $\si$-$\Delta$-closure of the projection of $G$ on the $i$-th factor
  $H_i(L_i)$.\\ 
Then there is a partition of
  $\{1,\ldots,r\}$ into subsets $I_1,\ldots,I_s$, such that for each
  $1\leq k\leq s$,  the following holds:\\
 If $i\neq j\in I_k$, then there are  an integer
  $n_{ij}\in\zee$ and an algebraic 
 isomorphism $\theta_{ij}:H_i(L_i)\to H_j(\si^{n_{ij}}(L_j))$ such that if $\pi_{I_k}$ is the projection $\prod_{j=1}^r
H_j(L_j)\to \prod_{j\in I_k}H_j(L_j)$, and $i\in I_k$ is fixed,
then $$\pi_{I_k}(G)=\{(g_j)_{j\in I_k}\in \prod_{j\in I_k}H_j(L_j)\mid \theta_{ij}(g_i)=
\si^{n_{ij}}(g_j) \hbox{ if }j\neq i\}. $$
Moreover, $G\simeq \prod_{k=1}^s\pi_{I_k}(G)$, and $G$ projects onto
each $\tilde G_i$.
\end{lem}

\prf 
We use the
prolongations $p_n$ defined in \ref{rem-LS1}, and choose $N$ large
enough so that $G=\{\bar g\in \prod_{i=1}^r H_i(L_i)\mid p_N(\bar g)\in G_{(N)}\}$. Then $G_{(N)}$
is a $\Delta$-algebraic subgroup of $$\prod_{i=1}^r (\tilde G_{i})_{(N)}\leq
\prod_{1\leq i\leq r, 0\leq k\leq N}H_i(\si^k(L_i)). $$
Let $\Psi\subset (\{1,\ldots,r\}\times \{0,\ldots,N\})^2$ be the set of pairs
given by Lemma \ref{lem3}, and $\psi_{(i,k),(j,\ell)}$, ${((i,k),(j,\ell))\in \Psi}$, the
corresponding set of algebraic isomorphisms $$\psi_{(i,k),(j,\ell)}:
H_i(\si^k(L_i))\to H_j(\si^\ell(L_j)).$$ So, if $(g_1,\ldots,g_r)\in G$, then
$$\psi_{(i,k),(j,\ell)}(\si^k(g_i))=\si^\ell(g_j).\eqno{(1)}$$
Note the following, whenever $((i,k),(j,\ell))\in\Psi$:
\begin{itemize}
\item If $k+1,\ell+1\leq N$, then $((i,k+1),(j,\ell+1))\in \Psi$, with
$\psi_{(i,k+1),(j,\ell+1) }={\psi_{(i,k),(j,\ell) }}^\si$ (here,
${\psi_{(i,k),(j,\ell) }}^\si$ denotes the isomorphism obtained by
applying $\si$ to the coefficients of the isomorphism
${\psi_{(i,k),(j,\ell) }})$; 
\item If $k,\ell\geq 1$, then $((i,k-1),(j,\ell-1))\in \Psi$, with
$\psi_{(i,k-1),(j,\ell-1) }={\psi_{(i,k),(j,\ell) }}^{\si\inv}$;
\item If  $k\leq \ell$, then
applying $\si^{-k}$ to  equation (1)   gives  $$((i,0),(j,\ell-k))\in
\Psi, \ \hbox{and}\ \psi_{(i,k),(j,\ell)}= {\psi_{(i,0),(j,\ell-k)}}^{\si^{k}}.$$
\item Finally, if $i=j$ and  $k<\ell$, then $\tilde G_i$ is defined by an equation
$\si^{n_i}(g)=\varphi_i(g)$ within $H_i(L_i)$ for some integer $n_i$ and
algebraic automorphism $
\varphi_i$ of $H(L)$, $((i,0),(i,n_i))\in \Psi$ with
associated isomorphism $\psi_{(i,0),(i,n_i)}=\varphi_i$, and $\ell-k$ is a 
multiple of the integer $n_i$. 
This is because  $G$ projects
onto a subgroup of finite index of $\tilde G_i$, and therefore the
$\Delta$-algebraic group $G_{(N)}$ 
projects onto the $\Delta$-algebraic group $(\tilde G_{i})_{(N)}$. However, from Fact \ref{simple}
and the definition of $\tilde G_i$ within $H_i(L)$, we deduce that $\tilde G_i$ has no
subgroup of finite index, and therefore that $G$ projects onto each $\tilde
G_i$. 
\end{itemize}
By Lemma \ref{lem3}, we know that the set $\Psi$ and the
 $\psi_{i,j}$ completely determine the $\tilde G_i$, and by the above observations,
 each condition  $\si^k(g_i)=\psi_{(i,k),(j,\ell)}(\si^\ell(g_j))$ is
 implied  by
 $$\si^{k-\ell}(g_i)={\psi_{(i,k),(j,\ell)}}^{\si^{-\ell}}(g_j).\eqno{(2)}$$  
The set $\Psi$ defines a structure of graph on $\{1,\ldots,r\}\times
\{0,\ldots,N\}$, which in turn induces a graph structure on
$\{1,\ldots,r\}$, by
$E(i,j)$ iff there are some $k,\ell$ such that $((i,k),(j,\ell))\in
\Psi$. If $E(i,j)$, then the isomorphism $\tilde G_i\to \tilde G_j$ is
given by equation (2). Then $(\{1,\ldots,r\},E)$ has finitely many
connected components, say $I_1,\ldots,I_s$, and for every $r$, if $i\in
I_r$, then $I_r=
\{i\}\cup \{j\mid E(i,j)\}$. 
Lemma \ref{lem3} shows that $G=\prod_{r=1}^s \pi_{I_r}(G)$, and gives
the desired description of $\pi_{I_r}(G)$, with
$\theta_{i,j}={\psi_{(i,k),(j,\ell)}}^{\si^{-\ell}}$ and
$n_{i,j}=k-\ell$, if $((i,k),(j,\ell))\in\Psi$ and $i\in I_r$. \qed

\begin{thm}\label{prop2semi} 
Let $G$ be a definable subgroup of $H(\calu)$, where $H$ is a
  semi-simple algebraic group defined and split over $\rat$, and with
  trivial center. 
   Assume that $G$ is Zariski dense in $H$. 
\begin{enumerate}
\item Assume that 
  the $\si$-$\Delta$-closure of $G$ is 
  connected for the $\si$-$\Delta$-topology. Then there are $s$ and
  simple normal algebraic subgroups $H_1,\ldots,H_s$ of $H$, a
  projection $\pi:H\to H_1\times \cdots \times H_s$ which restricts to an
  injective map on $G$,   fields of constants $L_i$ 
  in $\calu$, definable subgroups $G_i$ and
  $G'_i$ of $H_i(L_i)$ for $1\leq i\leq s$, and
  $h\in \pi(H)(\calu)$, such
  that $$ G_1\times \ldots \times G_s\leq h\inv \pi(G)h \leq G'_1\times
  \cdots \times G'_s,$$
  and each $G_i$ is a 
  normal subgroup of finite index of $G'_i$. 
\item  Assumptions as in (1). If  in addition $G$ is $\si$-$\Delta$-closed, then
  $h\inv \pi(G)h=G_1\times \cdots \times G_s$, and for each $i$, either $G_i=H_i(L_i)$, or for some
  integer $\ell_i$ and automorphism $\varphi_i$ of $H_i(L_i)$, $G_i$ is
  defined within $H_i(L_i)$ by $\si^{\ell_i}(g)=\varphi_i(g)$.

\end{enumerate}

\end{thm}

\prf 
By Theorem \ref{thm15}, if $H_1,\ldots,H_r$ are the simple
algebraic 
components of $H$, and $\bar G$ is the Kolchin closure of $G$, then
$\bar G$ is $\Delta$-semi-simple; if $\bar G_i$ is the connected (for the
$\Delta$-topology) component of
$\bar G\cap
H_i(\calu)$, then the morphism $\rho:\bar G_1\times \cdots\times \bar G_r\to
\bar G$ is an isogeny, and because $H$ is centerless, is an
isomorphism. 

\noindent
By Theorem \ref{Cassidy}, we know that there are
$\Delta$-definable
subfields $L_i$ of $\calu$, such that each $\bar G_i$ is conjugate to
$H_i(L_i)$ within $H_i(\calu)$. But as $[H_i,H_j]=1$ for $i\neq j$,
there is $h\in H(\calu)$ such that $h\inv \bar G_i h\leq
H_i(L_i)$ for all $i$. We will replace $G$ by $h\inv Gh$, so that 
$\bar G_i= H_i(L_i)$ for every $i$. \\[0.05in]
(1) 
For each $i$, consider the projection  $\pi_i$ on the $i$-th
factor $H_i(L_i)$, and  let $G'_i=\pi_i(G)$. Further,  let $G_i=H_i(L_i)\cap G$. So, $G_1\times
\cdots\times G_r$ is a subgroup of $G$. 

\medskip\noindent
{\bf Claim 1}.  $G'_i$ is Kolchin dense in $H_i(L_i)$, for
$i=1,\ldots,r$.

\prf 
Since $G$ is Kolchin dense in $\bar{G}$,   any generic $g:= (g_1,
\ldots, g_r)$ of $G$ is a generic of the $\Delta$-algebraic group $\bar G$.
Then  $g_i$ is a generic of $H_i(L_i)$ for all $i$, and the claim is
proved. \qed

  \medskip\noindent
 {\bf Claim 2}. For all $i \in \{1, \ldots, r\}$,  $G_i  \unlhd G_i'$.

\prf Let $q:H\to H_2\times \cdots \times H_r$ be the
  projection on the last $r-1$ factors.  Then  $G\cap \Ker(q)$ is
  normal in $G$, contained in $H_1(L_1)\times (1)^{r-1}$, and equals
  $G_1\times (1)^{r-1}$. As $G$
  projects onto $G'_1$, we get $G_1\unlhd G'_1$. The proof for the other indices is similar. \qed

  \medskip\noindent
  {\bf Claim 3}. If $G_i\neq (1)$, then $[G'_i:G_i]<\infty$. If moreover $G$ is
  quantifier-free definable, then $G_i=G'_i$. 

 \prf 
Both $G_i$ and $G'_i$ are definable subgroups of the simple
$\Delta$-algebraic group $H_i(L_i)$ and $G'_i$ is Kolchin dense in
$H_i(L_i)$. 

\noindent
If
$G'_i=H_i(L_i)$, then $G_i=G'_i$ since $H_i(L_i)$ is a simple (abstract)
group (by \ref{simple}, and because $Z(H)=(1)$). If $G'_i\neq H_i(L_i)$,
then by Theorem \ref{prop1}, Claim~1 and Lemma \ref{lem2},
$G'_i$ is definably quasi-simple. Moreover, since $G'_i$ is Zariski
dense in the centerless algebraic group $H_i$ (by Claim 1), it follows
that $G'_i$ has no finite normal subgroup, and so that any definable
normal subgroup of $G'_i$ must have finite index in $G_i$. Together with
Claim 2, this gives the 
result when $G$ is definable.\\
If $G$ is quantifier-free definable, so is every $G_i$, and therefore $G_i$
is closed in the $\si$-$\Delta$-topology. This implies that
$G_i=G'_i$, because $G$, and therefore also $G'_i$, is connected for the $\si$-$\Delta$-topology. \qed

\medskip\noindent 
If all $G_i$ are non-trivial,  we have shown that our group $G$ is squeezed between $G_1\times
\cdots \times G_r$ and $G'_1\times \cdots \times G'_r$. And that if $G$
is quantifier-free definable, then $G=\prod_{i=1}^r G_i$.\\[0.05in]
Assume now that some $G_i$ are trivial. If $\tilde G$ denotes the
$\si$-$\Delta$-closure of $G$, and $\tilde G_i$ the
$\si$-$\Delta$-closure of $G'_i$, then these groups are connected for
the $\si$-$\Delta$-topology, quantifier-free definable, and $\tilde G$
is a proper subgroup of $\prod_{i=1}^r \tilde G_i$. Hence Lemma
\ref{lem4} applies, and gives a subset $T$ of $\{1,\ldots,r\}$ such that
the natural projection $\pi_T$ defines an isomorphism $\tilde G\to
\prod_{i\in T}\tilde G_i$, which restricts to an embedding $G\to
\prod_{i\in T}G'_i$ with $[\prod_{i\in
  T}G'_i:\pi_T(G)]<\infty$. Moreover, applying Claim 3 to $G''_i:=\pi_T(G)\cap
H_i(L_i)$, $i\in T$, we get
$$\prod_{i\in T}G''_i\leq \pi_T(G)\leq \prod_{i\in T}G'_i,$$
with $G''_i$ a normal subgroup of $G'_i$ of finite index. This finishes
the proof of (1) (modulo a change of notation). \\
We showed that $\pi_T(\tilde G)=\prod_{i\in T}\tilde G_i$, which,
together with Theorem \ref{prop1}, proves (2). \qed

\begin{rem} In the case of $Z(H)\neq (1)$, we can obtain a
  similar result in a particular case: let $H_i(L_i)$ are
  the subgroups of $\bar G$ given by Theorem 
  \ref{thm15}, and  define $G_i=G\cap H_i(L_i)$ as above. Then if
  all $G_i$ are infinite or trivial, the same proof gives some 
  subset $T$ of $\{1,\ldots,r\}$, and an isogeny   $\prod_{i\in T}G_i$ onto a
  subgroup of finite index of $G$. \\
  In the general case, however, we can only obtain such a representation
  of a proper quotient of $G$: the problem arises from the fact that the
  groups $G_i$ may be finite non-trivial, so that the projection $\pi_T$ defined in
  the proof will restrict to an isogeny on $G$. So, we might as well
  work with the image of $G$ in $H/Z(H)$.

  \end{rem}

\section{Definable subgroups of finite index}

We work in  a sufficiently saturated model $(\calu, \si, \Delta)$ of
$\dcfa$. Unless
otherwise mentioned, definable will mean
$\call_{\si,\Delta}$-definable with parameters.\\
The aim of this section is to show that a definably quasi-simple group
definable in $\calu$ has a definable connected component. To do that, we
investigate definable subgroups of algebraic groups which are not
quantifier-free definable, and obtain a description similar to the one
obtained by Hrushovski and Pillay in Proposition 3.3 of \cite{HP2}.

  \begin{thm} \label{sbgp}  Let $H$ be an algebraic group, $G\leq
  H(\calu)$  a Zariski dense definable subgroup. Then there are an
  algebraic group $H'$, 
    a quantifier-free definable subgroup $R$ of $H'(\calu)$, together with a
  quantifier-free definable homomorphism $f:R\to G$, with $f(R)$ contained and of finite index in $G$, and $\Ker(f)$
    finite central in $R$.
\end{thm}

\prf We follow closely the proof of Hrushovski-Pillay given in
\cite[Prop.~3.3]{HP2}. Passing to a subgroup of $G$ of finite index, we may
assume that its $\sigma$-$\Delta$-closure $\tilde G$ is connected for the $\si$-$\Delta$-topology.  We
work over some small $F_0\prec \calu$ over which $G$ is defined.
By Theorem~\ref{LS1}(4), using compactness,  we know that there is some
quantifier-free 
 definable set $W$, and projection $\pi:W\to\tilde G$ with
finite fibers and such that $G=\pi(W)$; replacing if necessary $\pi$ by the transpose
of its graph, we
may assume that $\pi$ is the projection on
the first $t$ coordinates (where $t$ is such that $H\leq \aff^t$ or
${\mathbb P}^{t-1}$). \\[0.05in] 
Let $b,c$ be independent generics of $G$, let $a\in G$ be such that
$ab=c$, and let $\hat b,\hat c$   be tuples in $\calu$ such that
$(b,\hat b)$, $(c,\hat c)\in W$. So $\hat b \in \acl(F_0b)$, and $\hat c \in \acl(F_0c)$.
We let $a_1\in \calu$ be such that $\acl(F_0a)\cap F_0(b,\hat b,c,\hat
c)_{\si,\Delta}=F_0(a,a_1)_{\si,\Delta,}$. Note  that because $a=cb\inv$ and
$\acl(F_0a)$ is Galois over $F_0(a)_{\si,\Delta}$, 
the field $F_0(b,\hat b,c,\hat c)_{\si,\Delta}$ is a regular extension of
$\acl(F_0a)\cap F_0(b,\hat b,c,\hat c)_{\si,\Delta}$, and is finitely
generated algebraic over $F_0(a)_{\si,\Delta}$. Hence $a_1$ can
be chosen finite by Lemma \ref{cb1}.  Moreover, $qftp(b,\hat b,c,\hat
c/F_0(a,a_1)_{\si,\Delta})$ is stationary (see Remark  \ref{rem-LS1}(e)), and $F_0(a,a_1)_{\si,\Delta}$
contains $\qfcb(b,\hat b,c,\hat c/\acl(F_0a))$ (the quantifier-free
canonical basis, see subsection \ref{qfcb}).\\[0.05in]
Observe that $qftp(c,\hat c,a,a_1/F_0(b,\hat b)_{\si,\Delta})$ is stationary: this is
because \\$qftp(c,\hat c/F_0(b,\hat b)_{\si,\Delta})$ is stationary, and $(a,a_1)\in
F_0(b,\hat b,c,\hat c)_{\si,\Delta}$. We now define $b_1$ by  
$$\acl(F_0b)\cap
F_0(a,a_1,c,\hat c)_{\si,\Delta}=F_0(b,b_1)_{\si,\Delta},$$ and $c_1$ by 
 $$\acl(F_0c)\cap
F_0(a,a_1,b,b_1)_{\si,\Delta}=F_0(c,c_1)_{\si,\Delta}.$$

\noindent
By Remark \ref{canonicalbase}  we obtain :
\begin{itemize}
    \item $\qfcb(a,a_1,c,\hat c/\acl(F_0b))\subseteq F_0(b,b_1)_{\si,\Delta}$
\item $\qfcb(a,a_1,b,b_1/\acl(F_0c))\subseteq F_0(c,c_1)_{\si,\Delta}$.

\end{itemize}

\noindent
Since $b_1 \in \acl(F_0c)(a, a_1)_{\si, \Delta}$ and $c_1 \in \acl(F_0b)(a, a_1)_{\si, \Delta}$, we obtain that
$b_1\in F_0(a,a_1,c,c_1)_{\si,\Delta}$ and $a_1\in
F_0(b,b_1,c,c_1)_{\si,\Delta}$. I.e., we have
$$F_0(a,a_1,c,c_1)_{\si,\Delta}=F_0(a,a_1,b,b_1)_{\si,\Delta}=F_0(b,b_1,c,c_1)_{\si,\Delta}.\eqno{(1)}$$

\noindent
As in \cite{HP2}, $(a,a_1)$ defines the germ of a generically
defined and invertible $\si$-$\Delta$-rational map $g_{a,a_1}$ from (the
set of realisations of) $q_1=qftp(b,b_1/F_0)$ to
$q_2=qftp(c,c_1/F_0)$. (In our setting, this means:   there are
definable sets $U_1$ and $U_2$, such that the $\si$-$\Delta$-closure of
$U_i$ equals the $\si$-$\Delta$-closure of
the set of realisations of $q_i$, and such
that  $g_{a,a_1}$ defines a $\sigma$-$\Delta$-rational
invertible map $U_1\to U_2$. We may if necessary 
{replace} $U_i$ by a smaller set.) \\[0.05in]
Choose $(\tilde a,\tilde a_1)\in \calu$ realising $qftp(a,a_1/F_0)$ and
independent from $(a, b,c)$ over $F_0$. Let $F_0'\prec \calu$ contain
$F_0(\tilde a)$ and such that $(a,b,c)$ is independent from $F'_0$ over
$F_0$. \\[0.05in]
Let $(b',b'_1)$  be such that
$$qftp(a,a_1,b,b_1,c,c_1/F_0)=qftp(\tilde a,\tilde
a_1,b',b'_1,c,c_1/F_0).$$  Then we have
$$F_0(\tilde a,\tilde a_1,c,c_1)_{\si,\Delta}=F_0(\tilde a,\tilde
a_1,b',b'_1)_{\si,\Delta}=F_0(b',b'_1,c,c_1)_{\si,\Delta},\eqno{(1')}$$
so that $$F'_0(b',b'_1)_{\si,\Delta}=F'_0(c,c_1)_{\si,\Delta}\eqno{(2)}$$ because
$(b',b'_1)\in F_0(\tilde a, \tilde a_1, c,c_1)_{\si,\Delta}\subset 
F'_0(c,c_1)_{\si,\Delta}$. In particular, $$(b,b_1) \hbox{ and }
(b',b'_1) \hbox{  are
independent over }F'_0. \eqno{(3)}$$ (because $b$ and $c$ are).
We now let
$d=(\tilde a)\inv a$, and  
$r=qftp(a,a_1/F'_0)$ (the unique non-forking extension of
$qftp(a,a_1/F_0)$ to $F'_0$).   
	
	\medskip\noindent
{\bf Claim 1}.
	 \begin{enumerate}
    \item[(i)] $F'_0(b,\hat b,c,\hat c)_{\si,\Delta}\cap
      \acl(F'_0d)=F'_0(a,a_1)_{\si,\Delta}$.
      \item[(ii)] $qftp(b,b_1/F'_0)=qftp(b',b'_1/F'_0)=:q'_1$ is the unique
        non-forking extension of $q_1$ to $F'_0$.
        \item [(iii)] $(a,a_1)$ defines over $F'_0$ the germ of an
          invertible generically defined function from $q'_1$ to
          $q'_1$, sending $(b, b_1)$ to $(b', b_1')$.

        \item[(iv)]  $F_0'(a)_{\si, \Delta}= F_0'(d)_{\si, \Delta}.$
        \item[(v)] $db=b'$.
          \item[(vi)] $(a,a_1)\in F'_0(b,b_1,b',b'_1)_{\si,\Delta}$.
          \end{enumerate}

         \prf
   (iv) and (v) are clear from $d=(\tilde a)\inv a$ and $\tilde a\in F'_0$. Then
         (iv) and the fact that 
         $F'_0$ is independent from $(a,b,c)$ over $F_0$ give (i). We
         get (ii) from the independence of $(b,b_1)$ and $(b',b'_1)$
         over $F'_0$, that $(b,b_1)$ realises $q'_1$, and
         $qftp(b',b'_1/F_0)=qftp(b,b_1/F_0)$. Then (iii) follows from
         the fact that $g_{\tilde a,\tilde a_1}\inv$ is the germ of a
         function from $q_2$ to $q_1$, $(c,c_1)\mapsto (b',b'_1)$, and
         is defined over $F'_0$. Hence $h_{a,a_1}:=g_{\tilde a,\tilde a_1}\inv
         g_{(a,a_1)}$ is defined over $F'_0(a,a_1)_{\si,\Delta}$ and
         sends $(b,b_1)$ to $(b',b'_1)$.  Finally (vi) follows from
         equations (1) and ($1'$).    \qed

	\medskip\noindent
{\bf Claim 2}. The set of realisations of $r$ is closed under generic
composition, and if $(\alpha,\alpha _1)$ and $(a',a'_1)$ are $F'_0$-independent
realisations of $r$, then $F'_0(\alpha,\alpha_1,a',a'_1)_{\si,\Delta}$ contains a
realisation $(a'',a''_1)$ of $r$, which is independent from $(\alpha,\alpha_1)$
and from $(a',a'_1)$ over $F'_0$.  

\begin{proof}
 Note that if $A_1$, $A_2$, $A_3$, $A_4$ realise $q'_1$, and are such
that $A_1$ and $A_2$ are independent over $F'_0$ and $A_3$ and $A_4$ are
independent over $F'_0$, then
$qftp(A_1,A_2/F'_0)=qftp(A_3,A_4/F'_0)$. Since the assertion of the
claim only depends on
$qftp(\alpha,\alpha_1,a',a'_1/F'_0)$, we may assume  that
$(\alpha,\alpha_1)=(a,a_1)$, and assume that $(a',a'_1)$ is independent
from $(a,b,b')$ over $F'_0$, and this is what we will do. \\
If $(b'',b''_1)\in \calu$ is such that: $$qftp(
a',a'_1,b',b'_1,b'',b''_1/F'_0)=qftp(a,a_1,b,b_1,b',b'_1/F'_0),$$ then
from the fact that
$$F'_0(a,a_1,b,b_1)_{\si,\Delta}=F'_0(a,a_1,b',b'_1)_{\si,\Delta}=F'_0(b,b_1,b',b'_1)_{\si,\Delta}$$
(by equations (1) and ($1'$)), applying them to the tuple
$(a,a_1,b',b'_1,b'',b''_1)$ and using (3) we
obtain that $(b,b_1)$ and $(b'',b''_1)$ are independent over $F'_0$, and
realise $q'_1$. 

\noindent
So, if $(a'',a''_1)\in F'_0(b,b_1,b'',b''_1)_{\si,\Delta}$ is such that
$qftp(a'',a''_1,b,b_1,b'',b''_1/F'_0)=qftp(a,a_1,b,b_1,b',b'_1/F'_0)$,
then $qftp(a'',a''_1/F'_0)=r$. Because $b,b',b''$ are independent over
$F'_0$, we get that $(a'',a''_1)$ is independent from $(a,a_1)$ and from
$(a',a'_1)$ over $F'_0$. It remains to show that $(a'',a''_1)\in
F'_0(a,a_1,a',a'_1)_{\si,\Delta}$. From the definition of $h_{a,a_1}$ and
$h_{a',a'_1}$, we get $(\tilde a)\inv ab=b'$, $(\tilde a)\inv a'b'=b''$, and
therefore $(\tilde a)\inv a'\tilde a b=b''=(\tilde a)\inv a''b$, from
which we deduce that $a'(\tilde a)\inv a=a''$, so that $a''\in
F'_0(a,a')$. On the other hand, we know that $(b',b'_1)\in
F'_0(a,a_1,b,b_1)_{\si,\Delta}$, and that $(b'',b''_1)\in
F'_0(a',a'_1, b',b'_1)_{\si,\Delta}$, so that $$(a'',a''_1)\in
F'_0(a,a_1,a',a'_1, b,b_1)_{\si,\Delta}.$$  Hence  $$(a',a''_1)\in
\acl(F'_0,a,a')\cap F'_0(a,a_1,a',a'_1,
b,b_1)_{\si,\Delta}=F'_0(a,a_1,a',a'_1)_{\si,\Delta}$$
because $(b,b_1)$ is independent from $(a,a_1,a',a'_1)$ over $F'_0$, and
  the claim is proved. 
\end{proof}

\medskip\noindent
{\bf Claim 3}. $F'_0(a,a_1,a',a'_1)_{\si,\Delta}=F'_0(a,a_1, a'',
a''_1)_{\si,\Delta}= F'_0(a',a'_1, a'',a''_1)_{\si,\Delta}$.

\prf Let us consider the map $((a,a_1),(a',a'_1))\mapsto
(a'',a''_1)$, which is defined at least when $(a,a_1)$ and $(a',a'_1)$
are independent over $F'_0$. Observe now that
  $qftp(b,b_1,b',b'_1/F'_0)=qftp(b',b'_1,b,b_1/F'_0)$, and so we get a
  realisation $(\bar a,\bar a_1)$ of $r$ which represents the germ of
  the inverse of $(a,a_1)$; then $(\tilde a)\inv\bar a=a\inv\tilde a$, and therefore using
  (1),  $(\bar
  a,\bar a_1)\in \acl(F'_0 a)\cap  F'_0(a,a_1, b,
  b_1)_{\si,\Delta}=F'_0(a,a_1)_{\si,\Delta}$ (because $b$ is independent
  from $a$ over $F'_0$.)  
Reasoning as in the proof of Claim 2 and using  $(\tilde a)\inv a'=(\tilde
a)\inv a'' a\inv \tilde a$, and  $\tilde a \inv a=\tilde a (a')\inv
(\tilde a)\inv a''$, proves the claim. \qed

\noindent
Similarly,
using the fact that the first part of the tuple lives in the algebraic
group $H$, one gets that the map that  associates to $((a,a_1), (a',a'_1))$
 the tuple $(a'',a''_1)$ as above, is generically associative. Hence we are in
presence of a normal group law in the sense of \cite{W1} (pages 356-357), involving however
infinite tuples. \\[0.05in]
We now will reason as in \cite{P-Fund} (Lemma 2.3  and
Propositions 3.1 and 4.1 in \cite{P-Fund}):  as the $\si$-$\Delta$-topology is
Noetherian, we can find an algebraic group $H'$ and  a quantifier-free definable and connected subgroup $R$ of $H'(\calu)$ such that
$r$ is the generic type of $R$. 
More precisely: as in Lemma 2.3 of
  \cite{P-Fund}, we work in the pure field $\calu$, and replace
  $(a,a_1)$ by the infinite tuple obtained by closing $(a,a_1)$ under $\si$,
  $\si\inv$ and the $\delta_i$. This allows to represent the normal group law
  as a normal group law on some inverse limit of algebraic sets,
  together with a definable map from the set of realisations of $r$ to this
  inverse limit. (This map is the one which associates to a tuple $(g,g_1)$
  realising $r$ the infinite tuple of coordinates $(g,g_1)_{\si,\Delta}$. We
  will denote it by $\nabla$).  Then
  Proposition~3.1 of \cite{P-Fund} shows how to replace this inverse limit by an inverse
  limit of algebraic groups. And finally, as in Theorem~4.1 of \cite{P-Fund}, the
  Noetherianity of the $\si$-$\Delta$-topology guarantees that the map $\nabla$
  from the set of realizations of $r$ to this inverse limit of
$\calu$-points of algebraic  groups must yield an injection at some finite stage, say to  $H'(\calu)$. The algebraic 
  relations between the coordinates then give rise to the
  $\si$-$\Delta$-equations satisfied by $(g,g_1)$, i.e., they 
  quantifier-freely define the group $R\leq H'(\calu)$. 

 \noindent

\noindent
  Let us now look at $p=qftp(a,a_1,d/F'_0)$,
  and recall (by (iv) of Claim 1) that $F'_0(a)_{\si,\Delta}=F'_0(d)_{\si,\Delta}$, and let $K$ be the subgroup of
  $(H'\times H)(\calu)$ generated by the realisations of $p$. Observe that it is 
  definable by a quantifier-free $\call_{\si,\Delta}$-formula.\\[0.05in] 
  As in \cite{HP2}, it follows that $K$ is the graph of a group
  epimorphism $f:R\to \tilde G$ that is finite-to-one. Indeed, if $a$ is
  a generic of $G\leq H(\calu)$, then $f\inv(a)=(a,a_1)$ is algebraic
  over $F_0(a,a_1)_{\si,\Delta}$, so that $\SU(a/F_0)=\SU(a,a_1/F_0)$, i.e.,
  $\SU(R)=\SU(\tilde G)$;   hence $f(R)$ must have finite index
  in $\tilde G$. As $f\inv(g)$ is finite for every realisation of $p$,    $\Ker(f)$ must be
  finite. Because $R$ is
  connected for the $\si$-$\Delta$-topology, the kernel is central.  
	
	\medskip\noindent
{\bf Claim 4}. $f(R)\leq G$.
	
	\begin{proof}
	Let $(g,g_1)$ be a generic of $R$, i.e., a
  realisation of $r$. Then $g\in\tilde G$. We know that $qftp(b,\hat
  b,c, \hat c/F'_0(a,a_1)_{\si,\Delta})$ is stationary, and therefore so is its image
  under any $F'_0$-automorphism of the difference-differential field $\calu$
  sending $(a,a_1)$ to $(g,g_1)$, so that by Remark \ref{rem-LS1}(e), there are $(h,\hat h, u,\hat
  u)$ in $\calu$ such that
  $$qftp(a,a_1,b,\hat b,c,\hat c/F'_0)=qftp(g,g_1,h,\hat h,u,\hat
  u/F'_0).$$
  Thus $h,u\in G$, and so does $g=uh\inv$. \end{proof}

\noindent        
Observe that $f(R)$ has finite index in $G$, because it has finite index
in $\tilde G\geq G$. This finishes the proof of the theorem. \qed 

\begin{remark}\label{rem-sbgp} In the notation of Theorem \ref{sbgp},
consider $R_{(n)}$ and $G_{(n)}$, as well as the natural
$\call_\Delta$-map $f_{(n)}:R_{(n)}\to G_{(n)}$. While the map $f$ is
 not onto $\tilde G$, the
map $f_{(n)}$ is surjective onto $G_{(n)}$ for all $n\geq 0$ (in the differential field
$\calu$).  This follows from quantifier-elimination in $\dcf$. Moreover,
the image of $R$ in $\tilde G$ is dense for the $\si$-$\Delta$-topology, i.e., this
is the appropriate notion of a {\em dominant map} between
difference-differential varieties. 
\end{remark}

\begin{thm} \label{ACFAp} Let $\calu$ be a model of ACFA (possibly of positive
  characteristic $p$), let $H$ be an algebraic group defined over
  $\calu$, and $G\leq H(\calu)$ a definable subgroup, with $\si$-closure
  $\tilde G$. Then there is an algebraic group $H'$, a quantifier-free definable
  subgroup $R$ of $H'(\calu)$, and a quantifier-free definable
  homomorphism $f:R\to \tilde G$, with finite kernel and such that
  $[\tilde G:f(R)]$ is finite.
\end{thm}

\prf If the characteristic is $0$, this follows from Theorem \ref{sbgp}
with $m=0$. Suppose now that the characteristic is $p>0$, and let
$F_0\prec \calu$ be countable and contain the parameters needed to
define $H$ and $G$. \\[0.1in]
{\bf Claim}. There is a quantifier-free definable set $W$, and a
quantifier-free definable map $f: W\to \tilde G$, such that $f(W)=G$, and
for every 
$a\in G$, $f\inv(a)$ is finite, and separably algebraic over
$F_0(a)_\si$.

\prf We know this holds with $f\inv(a)$ algebraic over $F_0(a)_\si$ (by
Theorem \ref{LS1}(4)), we
need to show that we can find such an $f$ with $f\inv(a)$ separably
algebraic over $F_0(a)_\si$ when $a\in G$. This follows from the following
observation: $tp(a/F_0)$ is uniquely determined by the isomorphism type
over $F_0$ of the difference field $F_0(a)_\si^s$, the separable closure
  of the field $F_0(a)_\si$. Indeed, the automorphism $\si$ of $F_0(a)_\si^s$
  extends uniquely to $F_0(a)_\si^{alg}$. A compactness argument then
  shows the claim. \qed

\noindent
  Thus the proof of Theorem \ref{sbgp} can be reproduced almost verbatim,
using the following observation:\\[0.05in]
Given a difference subfield $A$ of $\calu$,  $\scl(A):=\acl(A)\cap A^s$ is Galois over
$A$, so that if $b\in\calu$, then $\scl(A)\cap A(b)_\si$  contains
  $\qfcb(b/\scl(A))$, because $A^s$ and $A(b)_\si$ are linearly disjoint
over their intersection. \\
Thus, replacing everywhere in the proof of \ref{sbgp} $\acl$ by $\scl$
gives the result. \qed

\bigskip\noindent
Another corollary worth mentioning is the following:
\begin{thm} \label{F-ell-2} Let $\calu\models \dcfa$, $\ell\geq 1$,  let $H$ be an algebraic group defined over
  $F_\ell$, and $G\leq H(F_\ell)$ a Zariski dense definable subgroup. Then there are an algebraic group $H'$, a quantifier-free $\call_{\si,\Delta}$-definable
  subgroup $R$ of $H'(F_\ell)$, and a quantifier-free definable
  homomorphism $f:R\to  G$, with finite kernel and such that
  $[G:f(R)]$ is finite.
\end{thm}

\prf Reason as in the proof of Theorem \ref{sbgp}, using Lemma
\ref{F-ell-1} to get $H'$ and $R$ definable in $F_\ell$. \qed

\begin{defn} Let $H$ be an algebraic group. It is {\em simply connected}
if it is connected and whenever $f:H'\to H$ is an isogeny from a connected algebraic group
$H'$ onto $H$, then $f$ is an isomorphism. \\
The {\em universal  covering of the connected algebraic group $H$} is a
simply connected algebraic group $\hat H$, together with an isogeny
$\pi:\hat H\to H$.  It satisfies the following universal property (see
18.8 in \cite{Mil}): if
$\varphi:H'\to H$ is an isogeny of  connected algebraic groups, 
then there is a
unique algebraic homomorphism $\psi: \hat H\to \hat H'$ such that
$\varphi \circ \psi=\pi$. 
\end{defn}

\begin{remark} \begin{enumerate}
    \item The definition of simply connected in arbitrary
  characteristic is a little 
  more complicated. The algebraic groups we will consider will be
  semi-simple algebraic groups, defined and split over $\rat$,  and we
  will be considering their rational points in some algebraically closed
  field $K$. 
\item  Every simple algebraic group has a universal  covering and it is itself a simple algebraic group, see
  section~5 in \cite{St62} for properties, or Chapter 19 in \cite{Mil}.
\item   Note that if $H$ is a simple algebraic group and $K$ is
algebraically closed, then $H(K)/Z(H(K))$ is simple as
an abstract group.
\item  Moreover, since a semi-simple algebraic group is isogenous to the
 product of its simple factors, it follows that the universal 
 covering of a semi-simple algebraic group is simply the product of the
 universal coverings of its simple factors.
 \end{enumerate}
\end{remark}
\begin{lem} \label{lem5} Let $H$ be a simple algebraic group defined
over the algebraically closed field $L$ of characteristic $0$, and
$\pi:\hat H\to H$ its
universal   covering. Then any algebraic automorphism of $H(L)$ lifts to
one of $\hat H(L)$.
\end{lem}

\begin{proof} Let $\varphi$ be an algebraic automorphism of $H(L)$, and consider the
map $p:\hat H(L) \to H(L)$ defined by $\varphi\circ \pi$. Then there is a map
$\psi:\hat H(L)\to \hat H(L)$ such that $\pi=\varphi\circ\pi\circ\psi$. It
then follows easily that $\psi$ is an isomorphism: $\psi(\hat H(L))$ is
a subgroup of $\hat H(L)$ which projects onto $H(L)$ via $\pi$, hence must
equal $\hat H(L)$. So $\psi$ is onto, and because $\Ker(\pi)$ is finite, it 
must be injective. 

\end{proof}
\begin{thm} \label{connected1} Let $H$ be a simply connected simple algebraic group defined
and split over $\rat$. Then every Zariski dense definable subgroup of
$H(\calu)$ is quantifier-free definable. 
Equivalently, every Zariski dense definable subgroup $G$ of 
$H(\calu)$ has a smallest definable subgroup $G^0$ of finite index, and $G^0$ is
quantifier-free definable. \\
Furthermore, if  $G\leq H(\calu)$ is definable, Zariski dense and
connected for the $\si$-$\Delta$-topology, then 
there is an $\call_\Delta$-definable subfield $L$ of
$\calu$, such that $h\inv G^0h\leq H(L)$ for some $h\in H(\calu)$, and
either $h\inv G^0h=H(L)$, or $$h\inv G^0h=\{g\in H(L)\mid \si^n(g)=\theta(g)\}$$ for
some integer $n$ and algebraic automorphism $\theta$ of $H(L)$. \end{thm}

\prf Let us first discuss the equivalence of the first two assertions. If any
Zariski dense definable subgroup of $H(\calu)$ is quantifier-free
definable, then 
the Noetherianity of the $\si$-$\Delta$ topology implies
that there is a smallest one, $G^0$. Conversely,  assume that any
Zariski dense definable subgroup $G$ of $H(\calu)$ has a smallest
definable subgroup of finite index, $G^0$, and that $G^0$ is
quantifier-free definable. Then so $G$, since it  is a finite union of
cosets of $G^0$. \\[0.05in]
So, we let $G\leq H(\calu)$ be Zariski dense in $H$, quantifier-free
definable, and connected for the
$\si$-$\Delta$-topology. Then the Kolchin closure $\bar G$ of $G$ is connected
for the Kolchin topology, and by Theorem \ref{Cassidy}, $\bar G$ is
conjugate to $H(L)$, for some field $L$ of constants in $\calu$. Hence
we may assume that $G\leq H(L)$. We want to show that $G$ has no
definable subgroup of finite index.\\
First note that if  $G=H(L)$, then
$G$ has no definable subgroups of finite index (by Fact \ref{simple}(a)), and the result is
proved. Assume therefore that $G$ is a proper subgroup of $H(L)$. 
Let $H'=H/Z(H)$. 
By Theorem \ref{prop1}, there are an integer $n\geq 1$ and an
algebraic automorphism $\theta'$ of $H'(L)$ such that the group 
$G' = GZ(H)/Z(H)$ is defined by 
$$G'=\{g\in H'(L)\mid \si^n(g)=\theta'(g)\}.$$
As $H$ is simply connected, $H\to H/Z(H)$ is the universal  covering of
$H/Z(H)$. By Lemma~\ref{lem5}, there is an algebraic automorphism $\theta$ of $H(L)$
which lifts $\theta'$.\\
{\bf Claim}. $G=\{g\in H(L)\mid \si^n(g)=\theta(g)\}$.\\
Indeed, the group on the right-hand side is $\si$-$\Delta$-closed and
connected (for the $\si$-$\Delta$-topology), and it contains $G$;
 as $G$ is a subgroup of finite index of $GZ(H)$, and is
$\si$-$\Delta$-closed, the equality follows. \qed

\noindent
Assume by way of contradiction that $G$ has a definable subgroup of
finite index $>1$. By Proposition \ref{sbgp}, there are a quantifier-free
definable group $R$ (living in some algebraic group $S$) and a
(quantifier-free) definable map $f:R\to G$ with finite non-trivial central
kernel, and image 
of finite index $>1$ in $G$. \\[0.05in]
For every $r\geq 1$, the map $f$ induces a  dominant $\Delta$-map
$f_{(r)}:R_{(r)}\to G_{(r)}$, and for $r\geq n-1$, this map has finite
central kernel, since for $r\geq n-1$, the natural map $G_{(r)}\to
G_{(n-1)}$ has trivial kernel. Observe that $f_{(r)}$ is surjective by  Remark \ref{rem-sbgp}.\\

\noindent
Fix $r\geq n-1$,
 and consider the (epimorphism) 
 $f_{(r)}:R_{(r)}\to
G_{(r)}\simeq H(L)^n$. Because $H$ is simply connected, so is $H^n$, and
therefore the map $f_{(r)}$ has an algebraic homomorphic section $g: G_{(r)}\to
R_{(r)}$ (i.e., $f_{(r)} g=id_{G_{(r)}}$); 
  since $G_{(r)}$ is connected, we obtain that $R_{(r)}\simeq H(L)^n\times
\Ker(f_{(r)})$. Since $H(L)$ equals its commutator subgroup, it follows that
$[R_{(r)},R_{(r)}]$ ($\simeq H(L)^n$) is 
a $\Delta$-definable normal subgroup of  $R_{(r)}$ which 
projects via $f_{(r)}$ onto $G_{(r)}\simeq H(L)^n$. As $R$ is
connected for the $\si$-$\Delta$-topology, $R_{(r)}$ is connected for
the $\Delta$-topology, and we must therefore have
$\Ker(f_{(r)})=(1)$. \\[0.05in]
So, we have shown that the  definable group $G$ is quantifier-free
definable and  has no proper definable subgroup of finite index. This
finishes the proof.\qed

\begin{thm} \label{prop2} Let $H$ be a simple algebraic group, $G\leq
H(\calu)$ be a definable subgroup which is Zariski dense in $H$. Then
$G$ has a smallest definable subgroup $G^0$ of finite index.  Let $\pi:\hat H\to H$ be the universal covering 
of $H$, and let $\tilde G$ be the connected component of the
$\si$-$\Delta$-closure of $\pi\inv(G)$. Then  $G^0=\pi(\tilde G)$. 
\end{thm}

\prf By Lemma \ref{connected1}, $\tilde
G$ has no definable subgroup of finite index. Hence, neither does
$\pi(\tilde G)$. As $\tilde G\cap G$ has finite index in $\pi(\tilde
G)$, we have that $\pi(\tilde G) \subseteq G$. We have shown that
$\pi(\tilde G)$ is contained in any Zariski dense definable subgroup of
$H$, therefore  $\pi(\tilde G)$ is  the smallest definable subgroup of
finite index of $G$. \qed

\begin{cor} \label{connected2} Let $G$ be a definably quasi-semi-simple definable
group. Then $G$ has a definable connected component. 
\end{cor}

\begin{proof} Let P be the property ``having a
definable connected component''. The result follows easily from Proposition \ref{thm1}, 
Proposition \ref{prop2semi}, Theorem \ref{prop2}, 
and the following remarks:
\begin{itemize}
\item[(a)] If $G_0$ is a definable subgroup of finite index of $G$, then $G_0$ has
P if and only if $G$ has $P$;
\item[(b)] If the group $G$ is the direct product of its definable subgroups $G_1$, $G_2$,
and $G_1$, $G_2$ have $P$, then so does $G$;
\item[(c)] Let $f:G\to G_1$ be a definable onto map, with $\Ker(f)$ finite. Then  $G_1$ has
P if and only if $G$ has $P$. One direction is clear: if $G$ has P, then
so does $G_1$, because an infinite strictly decreasing of definable
subgroups of $G_1$ of finite index would yield such a sequence for $G$;
for the other, assume that $G_1$ has P, and let $G_1^0$ be
its connected component; then $f\inv(G_1^0)$ is definable and has finite
index in $G$.  So, we may assume that $G_1$ is connected. \\
If $G_0$ is a definable subgroup of finite index of $G$,
then $f(G_0)=G_1$, so that $G_0\Ker(f)=G$; hence $[G:G_0]\leq
|\Ker(f)|$. This being true for any definable subgroup $G_0$ of $G$ of
finite index, implies that there is minimal one, and that its index
in $G$ is bounded by  $ |\Ker(f)|$. 
\end{itemize}
\end{proof}

\end{document}